\documentclass[a4paper]{cas-sc}

\usepackage{graphicx}
\usepackage{algorithm}
\usepackage{algorithmic}
\usepackage{amsthm}
\usepackage{placeins}

\usepackage{multirow}
\usepackage{makecell}
\usepackage{float}
\usepackage{tikz}

\newtheorem{method}{Method}[section]
\newtheorem{theorem}{Theorem}[section]
\newtheorem{lemma}{Lemma}[section]
\newtheorem{corollary}{Corollary}[section]
\newtheorem{remark}{Remark}[section]
\newtheorem{example}{Example}[section]
\newtheorem{definition}{Definition}[section]

\usepackage[numbers]{natbib}

\def\tsc#1{\csdef{#1}{\textsc{\lowercase{#1}}\xspace}}
\tsc{WGM}
\tsc{QE}

\AtBeginDocument{
	\setlength{\abovedisplayskip}{3pt}
	\setlength{\belowdisplayskip}{3pt}
	\setlength{\abovedisplayshortskip}{2pt}
	\setlength{\belowdisplayshortskip}{2pt}
}
\begin{document}
	\let\WriteBookmarks\relax
	\def\floatpagepagefraction{1}
	\def\textpagefraction{.001}

	\shorttitle{}    
	
	\shortauthors{A. K. S. Sengar, B. Kumar and Deepmala}  
	
	\title [mode = title]{A Relaxed Maximum-Based Normal \texorpdfstring{$S$}{S}-Iteration Method for Generalized Absolute Value Equations}  
	
	\tnotemark[1] 
	
	
	%
	
	\author[1]{Abhishek Kumar Singh Sengar}
	
	\cormark[1]
	
	
	\ead{abhishek47sengar@gmail.com}
	
	
	\credit{Conceptualization, Methodology, Software, Validation, Writing - original draft}
	
	\affiliation[1]{organization={Discipline of Natural Sciences},
		addressline={Indian Institute of Information Technology Design and Manufacturing}, 
		city={Jabalpur},
		postcode={482005}, 
		state={Madhya Pradesh},
		country={India}}
	
	\author[2]{Bharat Kumar}
	
	
	\ead{bharatnishad.kanpu@gmail.com}
	
	
	\credit{Formal analysis, Software, Validation, Writing - review and editing}
	
	\affiliation[2]{organization={Department of Economics},
		addressline={Indian Institute of Technology}, 
		city={Kanpur},
		postcode={208016}, 
		state={Uttar Pradesh},
		country={India}}

	\author[3]{Deepmala}
	
	
	\ead{dmrai23@gmail.com}
	
	
	\credit{Supervision, Formal analysis, Investigation, Writing - review and editing}
	
	\affiliation[3]{organization={Discipline of Natural Sciences},
		addressline={Indian Institute of Information Technology Design and Manufacturing}, 
		city={Jabalpur},
		postcode={482005}, 
		state={Madhya Pradesh},
		country={India}}

	\cortext[1]{Corresponding author}
	
	
	
\begin{abstract}
	A relaxed maximum-based normal $S$-iteration method (RMNSI) is proposed
	for solving generalized absolute value equations (GAVEs). The method combines
	a maximum-based fixed-point formulation with a constant relaxation
	parameter and avoids the selection of an auxiliary matrix. Assuming
	that $A+B$ is non-singular, both stages require linear systems with the
	same coefficient matrix $A+B$, allowing a single factorization to be
	reused throughout the iteration. A single spectral condition is established
	that guarantees unique solvability of the GAVE and global convergence
	from an arbitrary initial vector. An admissible range of the relaxation
	parameter is characterized, and $R$-linear convergence and error
	estimates are derived. Since the global condition can be conservative,
	a local convergence analysis based on the solution sign pattern is
	also presented. Numerical experiments on complementarity-derived, dense mixed-sign,
	and asymmetric ridge-regression problems are presented, and comparisons
	with several recent methods demonstrate the competitive efficiency and
	accuracy of RMNSI. The results further indicate that the preferred
	relaxation parameter depends mainly on the matrix structure and
	diagonal shift, while its dependence on the problem dimension is
	generally weak.
\end{abstract}

	

	\begin{keywords}
		Normal S-iteration \sep Maximum-based \sep Iterative Method \sep  Convergence Analysis \sep Generalized absolute value equations
	\end{keywords}
	
	\maketitle
	
	\section{Introduction}\label{sec1}
Following the equivalent sign convention adopted in this paper, the
generalized absolute value equation (GAVE) studied by Rohn
\cite{Rohn2004} is written as the weakly nonlinear system
\begin{equation}\label{eq:GAVE}
	Ax-B|x|=b,
\end{equation}
where $A,B\in\mathbb{R}^{n\times n}$,
$b\in\mathbb{R}^n$, and $x\in\mathbb{R}^n$ is the unknown vector. Here,
$|x|$ denotes the component-wise absolute value of $x$, which constitutes
a non-smooth piecewise-linear term. When $B=I_n$, where $I_n$ denotes
the identity matrix of order $n$, GAVE~\eqref{eq:GAVE} reduces to the
standard absolute value equation (AVE)
\begin{equation}\label{eq:AVE}
	Ax-|x|=b.
\end{equation}

The importance of absolute value equations arises from their close
connections with several problems in optimization and numerical
analysis \cite{Saad2003}. Mangasarian \cite{Mangasarian2014} established a relationship
between AVEs and linear complementarity problems, while Prokopyev
\cite{Prokopyev2009} investigated equivalent AVE reformulations and
their connections with mixed-integer linear programming. Mangasarian
\cite{Mangasarian2007} also developed the broader framework of absolute
value programming. More recently, GAVEs have appeared as first-order
optimality conditions for asymmetric ridge-regression models
\cite{XieQiHan2025}. These connections, together with the non-smooth and
potentially large-scale structure of the resulting systems, highlight
the need for efficient and reliable numerical methods.

The study of GAVEs has consequently progressed through two closely
connected strands. Theoretical investigations have focused primarily
on the existence, uniqueness, and characterization of solutions. These
include a theorem of the alternatives for uniform unique solvability
\cite{Rohn2004}, conditions for solvability or unsolvability
\cite{Mangasarian2017,Mezzadri2020}, criteria for unique solvability
\cite{Rohn2009,WuShen2021}, and properties of
the solution set \cite{Hladik2023}. Bounds for solutions, perturbation
estimates, error bounds, and condition numbers have also been studied
\cite{Hladik2018,ZamaniHladik2023,WuLi2025}. Comprehensive accounts of
these theoretical developments and the remaining challenges can be
found in \cite{KumarDeepmalaMoosaei2024,HladikEtAl2025}. The second
strand is computational and concerns the construction and convergence
analysis of numerical methods. The two strands complement each other:
solvability and matrix conditions support convergence analysis, while
error bounds provide tools for assessing approximations and justifying
practical stopping criteria.

A broad range of numerical methods has been developed for AVEs and
GAVEs. Newton-type approaches include the generalized Newton method
\cite{Mangasarian2009Newton}, globally and quadratically convergent
schemes \cite{CaccettaQuZhou2011}, inexact semismooth Newton methods
\cite{BelloCruzFerreiraPrudente2016}, and smoothing Newton algorithms
\cite{ChenYuHanMa2025}. Combining generalized Newton techniques with
matrix splitting, Li and Chen \cite{LiChen2025} proposed the GNMS method
for GAVEs and demonstrated its application to asymmetric ridge
regression. A concave-minimization approach has also been investigated
\cite{Mangasarian2007Concave}.

Fixed-point and matrix-splitting techniques form another major class.
Rohn et al.~\cite{RohnHooshyarbakhshFarhadsefat2014} studied the Picard
iteration, while Salkuyeh \cite{Salkuyeh2014} combined it with the HSS
splitting. Ke and Ma \cite{KeMa2017} introduced an SOR-like iteration,
Edalatpour et al.~\cite{EdalatpourHezariSalkuyeh2017} developed a
generalized Gauss-Seidel method for AVEs, and Dong et al.
\cite{DongShaoShen2020} subsequently proposed a new SOR-like scheme.
For GAVEs, Li et al.~\cite{LiLiDou2023} developed a shift-splitting
fixed-point method, while Zhang et al.
\cite{ZhangZhangLiangLiao2023} incorporated momentum acceleration into
a matrix-splitting iteration. Luo et al.~\cite{LuoEtAl2026} extended
the generalized Gauss-Seidel approach from AVEs to GAVEs, and Yu et
al.~\cite{YuEtAl2026} proposed an extended matrix-splitting iteration.

Different reformulations have led to exact and inexact
Douglas-Rachford splitting methods \cite{ChenYuHan2023}, an
inverse-free dynamical system \cite{ChenYangYuHan2021}, and an
alternating-projection method \cite{AlcantaraChenTam2023}. More
recently, Xie et al.~\cite{XieQiHan2025} introduced a randomized
iterative framework for square and rectangular GAVEs that contains the
Picard and randomized Kaczmarz iterations as particular cases.
	
Recently, Wu et al.~\cite{WuHanLi2024} introduced a new class of
maximum-based iterations for GAVEs. By rewriting the absolute-value
term through
$
|x|=2\max\{0,x\}-x,
$
they obtained a flexible framework involving a positive semi-definite
auxiliary matrix $\Omega$. Using the same maximum-based representation, Wang and Mu
\cite{WangMu2025} applied a splitting of $A+B$ to construct two
fixed-point methods. Fang et al.~\cite{FangWangQiao2025}
subsequently removed the auxiliary matrix and coupled $A-B$ and $A+B$
through a scalar relaxation parameter.

When $A+B$ is non-singular, the maximum-based reformulation of
GAVE~\eqref{eq:GAVE} defines a fixed-point operator. This structure
allows established fixed-point processes to be considered for its
solution. In particular, the normal $S$-iteration studied by Sahu
\cite{sahu2011applications} forms an intermediate point from the
current iterate and its image before evaluating the operator again.
To the best of our knowledge, normal $S$-iteration has not previously
been applied to solve GAVEs.

Bringing these ideas together, we propose the relaxed maximum-based
normal $S$-iteration method, abbreviated as RMNSI. By setting
$\Omega=O_n$, the auxiliary matrix is removed from the maximum-based
formulation. The resulting fixed-point operator is evaluated first at
the current iterate and then at a relaxed intermediate point. A
constant parameter $\theta\geq0$ controls this intermediate update and
permits both under-relaxed and over-relaxed choices. The main
contributions of this work are summarized as follows:

\begin{enumerate}
	\item The normal \(S\)-iteration process is used within the
	maximum-based fixed-point framework to construct RMNSI without requiring
	the selection of an auxiliary matrix. Both stages require linear systems with the same coefficient matrix $A+B$; therefore, a single factorization can be computed once
	and reused throughout the iteration.
	
\item Under the non-singularity of \(A+B\), a single spectral condition
is established that simultaneously guarantees the unique solvability of
the GAVE~\eqref{eq:GAVE} and the global convergence of RMNSI from any initial vector.
The admissible range of \(\theta\), \(R\)-linear convergence, and
residual- and iteration-based error estimates are derived. A local
convergence condition based on the solution sign pattern is also
obtained.
	
	\item Numerical experiments on LCP-derived, dense, and
	ridge-regression GAVEs are presented to compare RMNSI with recent
	methods. The effects of $\theta$, the diagonal shift, and the problem
	dimension are examined, and the observed behaviour is related to the
	global and local convergence conditions.
\end{enumerate}
	
The remainder of the paper is arranged as follows.
Section~\ref{sec2} collects the notation and preliminary results.
Section~\ref{sec3} develops the RMNSI method, and
Section~\ref{sec4} presents its convergence analysis.
Section~\ref{sec5} reports the numerical experiments, while
Section~\ref{sec6} concludes the paper.

\section{Preliminaries}\label{sec2}

This section introduces the notation and auxiliary results used in the
construction and convergence analysis of the proposed method. Vector
and matrix inequalities, as well as the maximum and minimum operators,
are understood component-wise. The symbols $|x|$ and $|G|$ denote the
component-wise absolute values of a vector $x$ and a matrix $G$,
respectively. For a square matrix $G$, its spectrum and spectral radius are denoted by
$\sigma(G)$ and $\rho(G)$, respectively. For
$x=(x_i)\in\mathbb{R}^n$ and
$G=(g_{ij})\in\mathbb{R}^{n\times n}$, the vector norms and the induced
matrix infinity norm used in this paper are defined by
$
	\|x\|_2
	=
	\left(\sum_{i=1}^n |x_i|^2\right)^{1/2},
	~
	\|x\|_\infty
	=
	\max_{1\leq i\leq n}|x_i|,
	~
	\|G\|_\infty
	=
	\max_{1\leq i\leq n}\sum_{j=1}^n|g_{ij}|.
$
Furthermore, $I_n$, $O_n$, and $\boldsymbol{e}$ denote the identity
matrix of order $n$, the zero matrix of order $n$, and the vector of all
ones in $\mathbb{R}^n$, respectively. All vector norms on
$\mathbb{R}^n$ are equivalent.

We first recall the definitions used in the subsequent analysis.

For a matrix $G=(g_{ij})\in\mathbb{R}^{n\times n}$, its comparison
matrix $\langle G\rangle=(\widehat g_{ij})$ is defined by
$
	\widehat g_{ij}
	=
	\begin{cases}
		|g_{ii}|, & i=j,\\
		-|g_{ij}|, & i\neq j.
	\end{cases}
$

A matrix is called a $Z$-matrix if all its off-diagonal entries are
non-positive. A $Z$-matrix is called a non-singular $M$-matrix if it is
non-singular and its inverse is non-negative. A matrix $G$ is called an
$H$-matrix if $\langle G\rangle$ is a non-singular $M$-matrix. A
splitting $P=M-N$ is called an $M$-splitting if $M$ is a non-singular
$M$-matrix and $N\geq O_n$; see
\cite{BermanPlemmons1994,Varga2000}.

\begin{definition}[{\cite{OrtegaRheinboldt2000}}]
	\label{def:R-linear}
	A sequence $\{x^{(k)}\}$ converging to $x^*$ is said to converge
	$R$-linearly if there exist constants $c>0$ and $r\in(0,1)$ such that
	\begin{equation*}
		\|x^{(k)}-x^*\|
		\leq
		c r^k,
		\qquad k=0,1,2,\ldots,
	\end{equation*}
	where $\|\cdot\|$ is any fixed vector norm on $\mathbb{R}^n$.
\end{definition}

\begin{definition}[Normal \(S\)-iteration
	\cite{sahu2011applications}]
	\label{def:normal-S}
	Let $\mathcal{C}$ be a nonempty convex subset of a normed linear space
	$X$, and let $T:\mathcal{C}\rightarrow\mathcal{C}$ be an operator.
	Given $x^{(0)}\in\mathcal{C}$, the normal $S$-iteration process is
	defined by
	\begin{equation*}
		x^{(k+1)}
		=
		T\left(
		(1-\lambda_k)x^{(k)}
		+\lambda_kT\bigl(x^{(k)}\bigr)
		\right),
		\qquad k=0,1,2,\ldots,
	\end{equation*}
	where $\{\lambda_k\}\subset(0,1)$.
\end{definition}

The following auxiliary results will be used in the construction and
convergence analysis of the proposed method.

\begin{lemma}[Perron-Frobenius theorem
	{\cite{BermanPlemmons1994}}]
	\label{lem:perron-frobenius}
	Let $G\in\mathbb{R}^{n\times n}$ be non-negative. Then
	$\rho(G)$ is an eigenvalue of $G$ and $G$ has a corresponding
	non-negative eigenvector. Moreover, if
	$
		O_n\leq G\leq H,
	$
	then
	$
		\rho(G)\leq\rho(H).
	$
\end{lemma}

\begin{lemma}[Standard spectral properties
	{\cite{BermanPlemmons1994,Horn2013}}]
	\label{lem:spectral-properties}
	Let \(G\in\mathbb{R}^{n\times n}\). For a polynomial
	\(p(t)=\sum_{j=0}^{d}a_jt^j\), define
	\(p(G)=\sum_{j=0}^{d}a_jG^j\), where \(G^0=I_n\), and let
	$
	p\bigl(\sigma(G)\bigr)
	=
	\{p(\lambda):\lambda\in\sigma(G)\}.
	$
	Then the following statements hold:
	\begin{enumerate}
		\renewcommand{\labelenumi}{(\roman{enumi})}
		
		\item For every polynomial \(p\),
		$
		\sigma\bigl(p(G)\bigr)
		=
		p\bigl(\sigma(G)\bigr).
		$
		
		\item For every \(s>\rho(G)\) and any fixed induced matrix norm, there
		exists a constant \(c_s>0\) such that
		$
		\|G^k\|
		\leq
		c_s s^k,
		~ k=0,1,2,\ldots.
		$
		
		\item If \(\rho(G)<1\), then there exists a vector norm
		\(\|\cdot\|_*\) on \(\mathbb{R}^n\) whose induced matrix norm, also
		denoted by \(\|\cdot\|_*\), satisfies
		$
		\|G\|_*<1.
		$
		
		\item For every induced matrix norm \(\|\cdot\|\),
		$
		\rho(G)\leq\|G\|.
		$
		
	\end{enumerate}
\end{lemma}

\begin{lemma}[{\cite{BermanPlemmons1994}}]
	\label{lem:nonnegative-matrix}
	Let $G\in\mathbb{R}^{n\times n}$ satisfy $G\geq O_n$ and
	$\rho(G)<1$. Then
$
		G^k\longrightarrow O_n,
	~
		(I_n-G)^{-1}
		=
		\sum_{k=0}^{\infty}G^k
		\geq O_n.
$
\end{lemma}

\begin{lemma}[{\cite{BermanPlemmons1994,Varga2000}}]
	\label{lem:M-matrix-properties}
	The following statements hold.
	\begin{enumerate}
		\renewcommand{\labelenumi}{(\roman{enumi})}
		
		\item If \(G\) is an \(H\)-matrix, then \(G\) is non-singular and
		$
		|G^{-1}|\leq\langle G\rangle^{-1}.
		$
		
		\item Let \(M\) be a \(Z\)-matrix and let \(N\geq O_n\). If \(M-N\) is
		a non-singular \(M\)-matrix, then \(M\) is a non-singular
		\(M\)-matrix and
		$
		\rho(M^{-1}N)<1.
		$
		
		\item A strictly row diagonally dominant \(Z\)-matrix with positive
		diagonal entries is a non-singular \(M\)-matrix.
		
	\end{enumerate}
\end{lemma}

\begin{lemma}\label{lem:max-nonexpansive}
	For any $x,y\in\mathbb{R}^n$,
	$
		\left|
		\max\{0,x\}-\max\{0,y\}
		\right|
		\leq
		|x-y|.
	$
\end{lemma}

\begin{proof}
	Let \(s,t\in\mathbb{R}\). If \(s,t\geq0\), then
	\(|\max\{0,s\}-\max\{0,t\}|=|s-t|\), whereas if \(s,t\leq0\), the
	left-hand side is zero. If \(s\geq0\geq t\), then
	\(|\max\{0,s\}-\max\{0,t\}|=s\leq s-t=|s-t|\); the remaining case is
	symmetric. Thus,
	\(|\max\{0,s\}-\max\{0,t\}|\leq|s-t|\) for all \(s,t\in\mathbb{R}\).
	Applying this inequality component-wise to \(x\) and \(y\) proves the
	result.
\end{proof}

\begin{lemma}[{\cite{WuHanLi2024}}]
	\label{lem:max-reformulation}
	For every $x\in\mathbb{R}^n$, GAVE~\eqref{eq:GAVE} is equivalent to
	$
		(A+B)x-2B\max\{0,x\}=b.
	$
\end{lemma}

\section{A relaxed maximum-based normal
	\texorpdfstring{$S$}{S}-iteration method for GAVE}\label{sec3}

In this section, we develop a relaxed maximum-based normal
$S$-iteration method for solving GAVE~\eqref{eq:GAVE}. As a starting
point, we briefly recall the maximum-based iteration method of
Wu et al.~\cite{WuHanLi2024}.

By Lemma~\ref{lem:max-reformulation}, GAVE~\eqref{eq:GAVE} can be
expressed as
\begin{equation*}
	(A+B)x-2B\max\{0,x\}=b.
\end{equation*}
Introducing an auxiliary matrix $\Omega\in\mathbb{R}^{n\times n}$ gives
the equivalent relation
\begin{equation*}
	(A+B+\Omega)x
	=
	\Omega x+2B\max\{0,x\}+b.
\end{equation*}
This formulation leads to the following method.

\begin{method}[Maximum-based iteration method
	\cite{WuHanLi2024}]\label{method:MBI}
	Let $x^{(0)}\in\mathbb{R}^n$ be an initial vector, and choose a positive semi-definite matrix $\Omega\in\mathbb{R}^{n\times n}$ such that
	$A+B+\Omega$ is non-singular. For $k=0,1,2,\ldots$, compute
	\begin{equation*}
		x^{(k+1)}
		=
		(A+B+\Omega)^{-1}
		\left(
		\Omega x^{(k)}
		+2B\max\{0,x^{(k)}\}+b
		\right)
	\end{equation*}
	until a prescribed stopping criterion is satisfied.
\end{method}

Method~\ref{method:MBI} provides a general maximum-based framework but
requires the selection of $\Omega$. To remove this
additional choice, we set $\Omega=O_n$ and incorporate the resulting fixed-point operator into a relaxed normal
$S$-iteration.

Assume that $A+B$ is non-singular. Taking $\Omega=O_n$ in
Method~\ref{method:MBI} gives the maximum-based operator
\begin{equation}\label{eq:max-map}
	T(x)
	=
	(A+B)^{-1}
	\left(
	2B\max\{0,x\}+b
	\right),
	\qquad x\in\mathbb{R}^n.
\end{equation}
It follows from Lemma~\ref{lem:max-reformulation} that
\begin{equation*}
	x^*\ \text{solves GAVE~\eqref{eq:GAVE}}
	\quad\Longleftrightarrow\quad
	T(x^*)=x^*.
\end{equation*}

In the present setting, $\mathcal{C}=\mathbb{R}^n$ is convex and
$T:\mathbb{R}^n\rightarrow\mathbb{R}^n$. Motivated by the normal
$S$-iteration in Definition~\ref{def:normal-S}, we replace the sequence
$\{\lambda_k\}\subset(0,1)$ by a constant relaxation parameter
$\theta\geq0$, subject to the convergence conditions established in
Section~\ref{sec4}. Define
\begin{equation}\label{eq:relaxed-S-map}
	S_\theta(x)
	=
	T\left((1-\theta)x+\theta T(x)\right).
\end{equation}
The proposed iteration can then be written as
$x^{(k+1)}=S_\theta(x^{(k)})$.

\begin{method}[Relaxed maximum-based normal
	$S$-iteration method]\label{method:RMNSI}
	Let $x^{(0)}\in\mathbb{R}^n$ be an initial vector, let $A+B$ be
	non-singular, and choose $\theta\geq0$. For $k=0,1,2,\ldots$, compute
\begin{equation}\label{eq:RMNSI}
	\left\{
	\begin{aligned}
		u^{(k)}
		&=
		T(x^{(k)})
		=
		(A+B)^{-1}
		\left(2B\max\{0,x^{(k)}\}+b\right),\\
		y^{(k)}
		&=
		(1-\theta)x^{(k)}+\theta u^{(k)},\\
		x^{(k+1)}
		&=
		T(y^{(k)})
		=
		(A+B)^{-1}
		\left(2B\max\{0,y^{(k)}\}+b\right).
	\end{aligned}
	\right.
\end{equation}
	Continue until the prescribed stopping criterion is satisfied.
\end{method}

Method~\ref{method:RMNSI} is called the relaxed maximum-based normal
\texorpdfstring{$S$}{S}-iteration method (RMNSI). If $x^*$ solves
GAVE~\eqref{eq:GAVE}, then $T(x^*)=x^*$ and
\begin{equation*}
	S_\theta(x^*)
	=
	T\left((1-\theta)x^*+\theta T(x^*)\right)
	=
	T(x^*)
	=
	x^*.
\end{equation*}
Thus, every solution of the GAVE is preserved by the proposed
iteration.

\begin{remark}\label{rem:theta-cases}
	The parameter $\theta$ determines the contribution of the first
	maximum-based evaluation. In particular,
	\begin{itemize}
		\item for $\theta=0$, RMNSI reduces to the maximum-based iteration
		with $\Omega=O_n$;
		\item for $0<\theta<1$, it becomes a constant-parameter normal
		$S$-iteration;
		\item for $\theta=1$, it performs two successive applications of
		the operator $T$;
		\item for $\theta>1$, it gives an over-relaxed normal $S$-iteration,
		provided that the corresponding convergence conditions hold.
	\end{itemize}
\end{remark}

In an implementation, the inverse of $A+B$ is not formed explicitly.
Instead, the two linear systems
\begin{equation*}
	\begin{aligned}
		(A+B)u^{(k)}
		&=2B\max\{0,x^{(k)}\}+b,\\
		(A+B)x^{(k+1)}
		&=2B\max\{0,y^{(k)}\}+b
	\end{aligned}
\end{equation*}
are solved. Since their coefficient matrix remains fixed, a single
factorization of $A+B$ can be reused throughout the iteration, thereby
avoiding repeated factorizations.

\section{Convergence analysis}\label{sec4}

This section studies the convergence properties of RMNSI. We first
establish a spectral condition for the unique solvability of
GAVE~\eqref{eq:GAVE} and the global convergence of the method from an
arbitrary initial vector. We then determine the admissible range of the
relaxation parameter, examine the convergence rate, and derive
additional sufficient conditions. Error estimates based on the GAVE
residual and successive iterates are also presented, followed by a
local convergence analysis near the solution.

\begin{theorem}\label{thm:global-convergence}
	Let $\theta\geq0$ and suppose that $C=A+B$ is a non-singular matrix. Define
	\begin{equation}\label{eq:global-quantities}
		\mathcal{K}
		=
		2\left|C^{-1}B\right|,
		\qquad
		\eta=\rho(\mathcal{K}),
	\end{equation}
	and
	\begin{equation}\label{eq:H-chi}
		\mathcal{H}_\theta
		=
		|1-\theta|\mathcal{K}
		+\theta\mathcal{K}^2,
		\qquad
		\chi_\theta
		=
		\rho(\mathcal{H}_\theta).
	\end{equation}
	Then
	\begin{equation}\label{eq:chi-expression}
		\chi_\theta
		=
		\eta\left(|1-\theta|+\theta\eta\right).
	\end{equation}
	If
	\begin{equation}\label{eq:main-convergence-condition}
		\chi_\theta<1,
	\end{equation}
	then GAVE~\eqref{eq:GAVE} has a unique solution $x^*$ and, for every
	$x^{(0)}\in\mathbb{R}^n$, the sequence generated by
	Method~\ref{method:RMNSI} converges to $x^*$.
\end{theorem}

\begin{proof}
	Since $\mathcal{K}\geq O_n$,
	Lemma~\ref{lem:perron-frobenius} ensures that
	$\eta=\rho(\mathcal{K})$ is an eigenvalue of $\mathcal{K}$. Consider
	the polynomial
	\begin{equation*}
		p(t)=|1-\theta|t+\theta t^2.
	\end{equation*}
	Since $\mathcal{H}_\theta=p(\mathcal{K})$,
	Lemma~\ref{lem:spectral-properties}(i) shows that $p(\eta)$ is an
	eigenvalue of $\mathcal{H}_\theta$. Moreover, for every
	$\lambda\in\sigma(\mathcal{K})$,
	\begin{equation*}
		|p(\lambda)|
		\leq
		|1-\theta||\lambda|+\theta|\lambda|^2
		\leq
		|1-\theta|\eta+\theta\eta^2
		=
		p(\eta).
	\end{equation*}
	Therefore,
$
		\rho(\mathcal{H}_\theta)
		=
		p(\eta)
		=
		\eta\left(|1-\theta|+\theta\eta\right),
$
	which proves \eqref{eq:chi-expression}.
	
	We next show that \eqref{eq:main-convergence-condition} implies
	$\eta<1$. If $\eta\geq1$, then
	\begin{equation*}
		|1-\theta|+\theta\eta
		\geq
		|1-\theta|+\theta
		\geq1.
	\end{equation*}
	It follows that $\chi_\theta\geq\eta\geq1$, contradicting
	\eqref{eq:main-convergence-condition}.
	
	Let $T$ be the maximum-based operator defined in
	\eqref{eq:max-map}. Lemma~\ref{lem:max-nonexpansive} gives
	\begin{equation}\label{eq:T-majorization}
		|T(x)-T(y)|
		\leq
		\mathcal{K}|x-y|,
		\qquad x,y\in\mathbb{R}^n.
	\end{equation}
	Starting from an arbitrary $z^{(0)}\in\mathbb{R}^n$, define
	\begin{equation*}
		z^{(k+1)}=T(z^{(k)}),
		\qquad k=0,1,2,\ldots.
	\end{equation*}
	Repeated application of \eqref{eq:T-majorization} gives
	\begin{equation*}
		|z^{(k+1)}-z^{(k)}|
		\leq
		\mathcal{K}^k|z^{(1)}-z^{(0)}|.
	\end{equation*}
	Hence, for every positive integer $\ell$,
	\begin{align*}
		|z^{(k+\ell)}-z^{(k)}|
		&\leq
		\sum_{j=0}^{\ell-1}
		|z^{(k+j+1)}-z^{(k+j)}|
		\leq
		\sum_{j=0}^{\ell-1}
		\mathcal{K}^{k+j}|z^{(1)}-z^{(0)}|\\
		&\leq
		\mathcal{K}^k
		(I_n-\mathcal{K})^{-1}
		|z^{(1)}-z^{(0)}|.
	\end{align*}
	By Lemma~\ref{lem:nonnegative-matrix},
$
		(I_n-\mathcal{K})^{-1}
		=
		\sum_{j=0}^{\infty}\mathcal{K}^j
		\geq O_n,
		~
		\mathcal{K}^k\longrightarrow O_n.
$

	The preceding bound is independent of $\ell$. Taking the infinity norm
	and letting $k\to\infty$ show that $\{z^{(k)}\}$ is a Cauchy sequence and converges to some $x^*\in\mathbb{R}^n$. Since $T$ is
	continuous, passing to the limit in
	$z^{(k+1)}=T(z^{(k)})$ gives
$
		T(x^*)=x^*.
$

	By Lemma~\ref{lem:max-reformulation}, $x^*$ is a solution of
	GAVE~\eqref{eq:GAVE}.
	
	To establish uniqueness, let $\widehat{x}^*$ be another solution. Then
	$T(\widehat{x}^*)=\widehat{x}^*$, and
	\begin{equation*}
		|x^*-\widehat{x}^*|
		\leq
		\mathcal{K}|x^*-\widehat{x}^*|.
	\end{equation*}
	Iterating this inequality gives
	\begin{equation*}
		|x^*-\widehat{x}^*|
		\leq
		\mathcal{K}^{s}|x^*-\widehat{x}^*|,
		\qquad s=1,2,\ldots.
	\end{equation*}
	Letting $s\to\infty$ and using
	Lemma~\ref{lem:nonnegative-matrix} yields
	$x^*=\widehat{x}^*$. Hence, the solution is unique.
	
	It remains to prove the convergence of RMNSI. Since $T(x^*)=x^*$,
	the first relation in \eqref{eq:RMNSI}, together with
	\eqref{eq:T-majorization}, gives
	\begin{equation}\label{eq:u-error}
		|u^{(k)}-x^*|
		=
		|T(x^{(k)})-T(x^*)|
		\leq
		\mathcal{K}|x^{(k)}-x^*|.
	\end{equation}
	Using the second relation in \eqref{eq:RMNSI} and
	\eqref{eq:u-error}, we obtain
	\begin{equation}\label{eq:y-error}
		\begin{aligned}
			|y^{(k)}-x^*|
			&=
			\left|
			(1-\theta)(x^{(k)}-x^*)
			+\theta(u^{(k)}-x^*)
			\right|\\
			&\leq
			\left(
			|1-\theta|I_n+\theta\mathcal{K}
			\right)|x^{(k)}-x^*|.
		\end{aligned}
	\end{equation}
	Finally, the third relation in \eqref{eq:RMNSI}, together with
	\eqref{eq:T-majorization} and \eqref{eq:y-error}, yields
	\begin{equation}\label{eq:RMNSI-error}
		|x^{(k+1)}-x^*|
		=
		|T(y^{(k)})-T(x^*)|
		\leq
		\mathcal{K}|y^{(k)}-x^*|
		\leq
		\mathcal{H}_\theta|x^{(k)}-x^*|.
	\end{equation}
	Repeated application of \eqref{eq:RMNSI-error} gives
	\begin{equation}\label{eq:global-error-bound}
		|x^{(k)}-x^*|
		\leq
		\mathcal{H}_\theta^k|x^{(0)}-x^*|.
	\end{equation}
	Since $\mathcal{H}_\theta\geq O_n$ and
	$\rho(\mathcal{H}_\theta)=\chi_\theta<1$,
	Lemma~\ref{lem:nonnegative-matrix} implies that
	$\mathcal{H}_\theta^k\to O_n$. Therefore,
	$x^{(k)}\to x^*$ for every $x^{(0)}\in\mathbb{R}^n$.
\end{proof}

The following corollary gives the range of $\theta$ for which the convergence condition in Theorem~\ref{thm:global-convergence} holds.
	
\begin{corollary}\label{cor:theta-range}
	Under the hypotheses and notation of
	Theorem~\ref{thm:global-convergence}, the condition
	$\chi_\theta<1$ is characterized as follows:
	\begin{enumerate}
		\item if $\eta=0$, every $\theta\geq0$ is admissible;
		\item if $0<\eta<1$, the condition holds if and only if
		\begin{equation}\label{eq:theta-range}
			0\leq\theta<\frac{1}{\eta};
		\end{equation}
		\item if $\eta\geq1$, no $\theta\geq0$ is admissible.
	\end{enumerate}
\end{corollary}

\begin{proof}
	If $\eta=0$, then \eqref{eq:chi-expression} immediately gives
	$\chi_\theta=0$ for every $\theta\geq0$.
	
	Suppose now that $0<\eta<1$. For $0\leq\theta\leq1$, we have
$
		\chi_\theta
		=
		\eta\left(1-\theta+\theta\eta\right)
		\leq\eta<1.
$
	For $\theta>1$, equation~\eqref{eq:chi-expression} gives
$
		\chi_\theta<1
	~\Longleftrightarrow~
		\eta\left(\theta-1+\theta\eta\right)<1
		~\Longleftrightarrow~
		\theta<\frac{1}{\eta}.
$
	Since $1/\eta>1$, these two intervals combine to give
	\eqref{eq:theta-range}.
	
	Finally, let $\eta\geq1$. Since $\theta\geq0$,
$
		\chi_\theta
		=
		\eta\left(|1-\theta|+\theta\eta\right)
		\geq
		\eta\left(|1-\theta|+\theta\right)
		\geq1.
$
	Hence, the convergence condition cannot hold for any
	$\theta\geq0$.
\end{proof}

The following corollary connects the unknown solution error with the
computable residual of GAVE~\eqref{eq:GAVE}.

	\begin{corollary}\label{cor:residual-error-bound}
		Under the assumptions of
		Theorem~\ref{thm:global-convergence}, the unique solution $x^*$ of
		GAVE~\eqref{eq:GAVE} satisfies
		\begin{equation}\label{eq:residual-error-bound}
			|x-x^*|
			\leq
			(I_n-\mathcal{K})^{-1}
			\left|
			C^{-1}\left(Ax-B|x|-b\right)
			\right|
		\end{equation}
		for every $x\in\mathbb{R}^n$.
	\end{corollary}
	
\begin{proof}
	Since $T(x^*)=x^*$, inequality~\eqref{eq:T-majorization} gives
	\begin{align*}
		|x-x^*|
		&\leq
		|x-T(x)|+|T(x)-T(x^*)|\\
		&\leq
		|x-T(x)|+\mathcal{K}|x-x^*|.
	\end{align*}
	Therefore,
	\begin{equation}\label{eq:residual-intermediate}
		(I_n-\mathcal{K})|x-x^*|
		\leq
		|x-T(x)|.
	\end{equation}
	By Theorem~\ref{thm:global-convergence},
	$\rho(\mathcal{K})=\eta<1$. Since $\mathcal{K}\geq O_n$,
	Lemma~\ref{lem:nonnegative-matrix} gives
	\begin{equation}\label{eq:K-Neumann-series}
		(I_n-\mathcal{K})^{-1}
		=
		\sum_{j=0}^{\infty}\mathcal{K}^j
		\geq O_n.
	\end{equation}
	Multiplying \eqref{eq:residual-intermediate} by the non-negative matrix
	in \eqref{eq:K-Neumann-series} yields
	\begin{equation*}
		|x-x^*|
		\leq
		(I_n-\mathcal{K})^{-1}|x-T(x)|.
	\end{equation*}
	Moreover, from the definition of $T$ and the identity
	$2\max\{0,x\}=x+|x|$, we have
	\begin{align*}
		x-T(x)
		&=
		C^{-1}
		\left(
		Cx-2B\max\{0,x\}-b
		\right)\\
		&=
		C^{-1}
		\left(
		Ax-B|x|-b
		\right).
	\end{align*}
	Substitution into the preceding inequality proves
	\eqref{eq:residual-error-bound}.
\end{proof}
	
	\begin{remark}\label{rem:residual-bound}
		Corollary~\ref{cor:residual-error-bound} is useful when the exact
		solution is unknown. The residual
		$Ax-B|x|-b$ can be evaluated directly from an approximation $x$, and
		\eqref{eq:residual-error-bound} shows that, under the stated
		assumptions, a small residual gives a controlled solution error. It
		therefore provides theoretical support for the residual-based stopping
		criterion used in the numerical experiments. The estimate is primarily
		theoretical, since explicitly computing the inverse matrices appearing
		in the bound is not required in the implementation.
	\end{remark}
	
We next derive two error estimates for RMNSI. The first measures the
error using the initial iterate, while the second uses the difference
between two successive iterates.
	
	\begin{corollary}\label{cor:iteration-error-bounds}
		Under the assumptions of
		Theorem~\ref{thm:global-convergence}, the RMNSI iterates satisfy
		\begin{equation}\label{eq:apriori-error}
			|x^{(k)}-x^*|
			\leq
			\mathcal{H}_\theta^k
			(I_n-\mathcal{H}_\theta)^{-1}
			|x^{(1)}-x^{(0)}|
		\end{equation}
		and
		\begin{equation}\label{eq:aposteriori-error}
			|x^{(k)}-x^*|
			\leq
			(I_n-\mathcal{H}_\theta)^{-1}
			|x^{(k+1)}-x^{(k)}|.
		\end{equation}
	\end{corollary}
	
\begin{proof}
	The argument used to obtain \eqref{eq:RMNSI-error} also shows that the
	RMNSI iteration map $S_\theta$ satisfies
	\begin{equation*}
		|S_\theta(x)-S_\theta(y)|
		\leq
		\mathcal{H}_\theta|x-y|,
		\qquad x,y\in\mathbb{R}^n.
	\end{equation*}
	Consequently,
	\begin{equation*}
		|x^{(k+j+1)}-x^{(k+j)}|
		\leq
		\mathcal{H}_\theta^j
		|x^{(k+1)}-x^{(k)}|,
		\qquad j\geq0.
	\end{equation*}
	
	By Theorem~\ref{thm:global-convergence},
	$x^{(k+j)}\to x^*$ as $j\to\infty$. Furthermore,
	$\mathcal{H}_\theta\geq O_n$ and
	$\rho(\mathcal{H}_\theta)<1$. Hence,
	Lemma~\ref{lem:nonnegative-matrix} gives
	\begin{equation*}
		(I_n-\mathcal{H}_\theta)^{-1}
		=
		\sum_{j=0}^{\infty}\mathcal{H}_\theta^j
		\geq O_n.
	\end{equation*}
	Summing the successive differences therefore yields
\begin{equation*}
	|x^*-x^{(k)}|
	\leq \sum_{j=0}^{\infty}|x^{(k+j+1)}-x^{(k+j)}|
	\leq \sum_{j=0}^{\infty}\mathcal{H}_\theta^j|x^{(k+1)}-x^{(k)}|
	=(I_n-\mathcal{H}_\theta)^{-1}|x^{(k+1)}-x^{(k)}|.
\end{equation*}
	This proves \eqref{eq:aposteriori-error}.
	
Using this inequality repeatedly gives
	\begin{equation*}
		|x^{(k+1)}-x^{(k)}|
		\leq
		\mathcal{H}_\theta^k
		|x^{(1)}-x^{(0)}|.
	\end{equation*}
	Substitution into \eqref{eq:aposteriori-error} yields
	\begin{equation*}
		|x^{(k)}-x^*|
		\leq
		(I_n-\mathcal{H}_\theta)^{-1}
		\mathcal{H}_\theta^k
		|x^{(1)}-x^{(0)}|.
	\end{equation*}
	Since $(I_n-\mathcal{H}_\theta)^{-1}$ is represented by a power
	series in $\mathcal{H}_\theta$, it commutes with
	$\mathcal{H}_\theta^k$. Hence,
	\begin{equation*}
		|x^{(k)}-x^*|
		\leq
		\mathcal{H}_\theta^k
		(I_n-\mathcal{H}_\theta)^{-1}
		|x^{(1)}-x^{(0)}|,
	\end{equation*}
	which proves \eqref{eq:apriori-error}.
\end{proof}

\begin{remark}\label{rem:iteration-error-bounds}
	Estimate~\eqref{eq:apriori-error} controls the error after $k$
	iterations using the initial difference $|x^{(1)}-x^{(0)}|$.
	Estimate~\eqref{eq:aposteriori-error}, on the other hand, uses the
	currently available difference $|x^{(k+1)}-x^{(k)}|$. It shows that
	small changes between successive iterates give a controlled solution
	error and can therefore guide a practical stopping criterion. The
	inverse appearing in these estimates is needed only for the analysis
	and is not formed during the iteration.
\end{remark}

	The global error estimate further gives a linear convergence result for
	RMNSI.
\begin{corollary}\label{cor:linear-convergence}
	Under the assumptions of
	Theorem~\ref{thm:global-convergence}, RMNSI is $R$-linearly
	convergent. More precisely, for every $r$ satisfying
	$
	\chi_\theta<r<1,
	$
	there exists a constant $c_r>0$ such that
	\begin{equation}\label{eq:R-linear-rate}
		\|x^{(k)}-x^*\|
		\leq
		c_r r^k\|x^{(0)}-x^*\|,
		\qquad k=0,1,2,\ldots,
	\end{equation}
	where $\|\cdot\|$ denotes any fixed vector norm on
	$\mathbb{R}^n$.
\end{corollary}

\begin{proof}
	Since
	$\rho(\mathcal{H}_\theta)=\chi_\theta<r$,
	Lemma~\ref{lem:spectral-properties}(ii) gives
	\begin{equation*}
		\|\mathcal{H}_\theta^k\|_\infty
		\leq
		c_r r^k,
		\qquad k=0,1,2,\ldots,
	\end{equation*}
	for some constant $c_r>0$. Taking the infinity norm in
	\eqref{eq:global-error-bound} yields
	\begin{equation*}
		\|x^{(k)}-x^*\|_\infty
		\leq
		\|\mathcal{H}_\theta^k\|_\infty
		\|x^{(0)}-x^*\|_\infty
		\leq
		c_r r^k\|x^{(0)}-x^*\|_\infty.
	\end{equation*}
	Definition~\ref{def:R-linear} therefore shows that the sequence
	converges $R$-linearly to $x^*$ in the infinity norm. Since all norms
	on $\mathbb{R}^n$ are equivalent, the same conclusion holds for any
	fixed vector norm, with a possibly different constant.
\end{proof}

The value of $\theta$ plays an important role in the convergence of
RMNSI. The following result describes its effect on the convergence
bound and compares RMNSI with the basic maximum-based iteration.

\begin{theorem}\label{thm:convergence-rate}
	Let $0<\eta<1$ and $0\leq\theta<1/\eta$. Then
	$\chi_\theta$ is strictly decreasing on $[0,1]$ and strictly increasing
	on $[1,1/\eta)$. Hence, its minimum is attained at
	\begin{equation}\label{eq:global-optimal-theta}
		\theta=1,
		\qquad
		\chi_1=\eta^2.
	\end{equation}
	Moreover, the RMNSI bound $\chi_\theta$ is smaller than the bound
	$\eta$ of the basic maximum-based iteration if and only if
	\begin{equation}\label{eq:improvement-range}
		0<\theta<\frac{2}{1+\eta}.
	\end{equation}
\end{theorem}

\begin{proof}
	For $0\leq\theta\leq1$, equation~\eqref{eq:chi-expression} gives
$
		\chi_\theta
		=
		\eta\left(1-\theta+\theta\eta\right)
		=
		\eta-\eta(1-\eta)\theta.
$
	Since $0<\eta<1$, this expression decreases as $\theta$ increases.
	For $1\leq\theta<1/\eta$, we have
$
		\chi_\theta
		=
		\eta\left(\theta-1+\theta\eta\right)
		=
		\eta(1+\eta)\theta-\eta,
$
	which increases with $\theta$. Therefore, $\chi_\theta$ is minimized
	at $\theta=1$, and its minimum value is $\eta^2$.
	
	The basic maximum-based iteration has the bound $\eta$. Hence, RMNSI
	has a smaller bound when
	\begin{equation*}
		\chi_\theta<\eta
		\quad\Longleftrightarrow\quad
		|1-\theta|+\theta\eta<1.
	\end{equation*}
	For $0\leq\theta\leq1$, this condition holds whenever $\theta>0$.
	For $\theta>1$, it is equivalent to
	$\theta<2/(1+\eta)$. Combining the two cases gives
	\eqref{eq:improvement-range}.
\end{proof}

\begin{remark}\label{rem:global-optimality}
	Theorem~\ref{thm:convergence-rate} identifies $\theta=1$ as the
	minimizer of the global upper bound $\chi_\theta$. This does not
	imply that $\theta=1$ minimizes the observed iteration count, because
	the bound is based on component-wise absolute values and can be
	conservative. In particular, it does not account for sign
	cancellation or the active sign pattern near the solution. The local
	analysis given below shows that the preferred value of $\theta$
	depends on the spectrum of the local iteration matrix and can
	therefore differ from one.
\end{remark}

	Computing the spectral radius in
	Theorem~\ref{thm:global-convergence} can be expensive for large
	problems. The following corollary gives a simpler condition using the
	matrix infinity norm.

	\begin{corollary}\label{cor:norm-condition}
		Let $\theta\geq0$ and suppose that $C=A+B$ is non-singular. Define
		\begin{equation}\label{eq:xi-definition}
			\xi
			=
			2\left\|C^{-1}B\right\|_\infty.
		\end{equation}
		If
		\begin{equation}\label{eq:norm-convergence-condition}
			\xi\left(|1-\theta|+\theta\xi\right)<1,
		\end{equation}
		then GAVE~\eqref{eq:GAVE} has a unique solution, and the sequence
		generated by Method~\ref{method:RMNSI} converges to it for every
		$x^{(0)}\in\mathbb{R}^n$.
	\end{corollary}
	
\begin{proof}
	Since
	$
		\mathcal{K}=2|C^{-1}B|,
	$
	Lemma~\ref{lem:spectral-properties}(iv) gives
$
		\eta
		=
		\rho(\mathcal{K})
		\leq
		\|\mathcal{K}\|_\infty
		=
		2\|C^{-1}B\|_\infty
		=
		\xi.
$

	The function
	\begin{equation*}
		t\longmapsto t\left(|1-\theta|+\theta t\right)
	\end{equation*}
	is non-decreasing for $t\geq0$. Therefore,
	\begin{equation*}
		\chi_\theta
		=
		\eta\left(|1-\theta|+\theta\eta\right)
		\leq
		\xi\left(|1-\theta|+\theta\xi\right)
		<1.
	\end{equation*}
	The conclusion follows from
	Theorem~\ref{thm:global-convergence}.
\end{proof}

We next establish a structural convergence condition based on
$M$-matrices. This condition avoids the direct use of
$\mathcal{K}= 2 \left|(A+B)^{-1}B\right|$ and also provides an admissible range for
the relaxation parameter.

\begin{theorem}\label{thm:M-matrix-condition}
		Let $C=A+B$ and suppose that
		\begin{equation}\label{eq:M-matrix-condition}
			\langle C\rangle-2|B|
		\end{equation}
		is a non-singular $M$-matrix. Define
		\begin{equation}\label{eq:eta-hat}
			\widehat{\eta}
			=
			\rho\left(
			2\langle C\rangle^{-1}|B|
			\right).
		\end{equation}
		Then $C$ is non-singular and $0\leq\widehat{\eta}<1$. Moreover, RMNSI
		converges to the unique solution of GAVE~\eqref{eq:GAVE} for every
		initial vector whenever
		\begin{equation}\label{eq:theta-M-matrix}
			0\leq\theta<\frac{1}{\widehat{\eta}}.
		\end{equation}
		If $\widehat{\eta}=0$, the upper bound in
		\eqref{eq:theta-M-matrix} is understood as infinity. In particular,
		every $\theta\in[0,1]$ is admissible.
	\end{theorem}
	
\begin{proof}
	Set
	$
		\mathcal{M}=\langle C\rangle,
		~
		\mathcal{N}=2|B|.
	$
	The matrix $\mathcal{M}$ is a $Z$-matrix, and
	$\mathcal{N}\geq O_n$. Since
	$\mathcal{M}-\mathcal{N}$ is a non-singular $M$-matrix,
	Lemma~\ref{lem:M-matrix-properties}(ii) shows that $\mathcal{M}$ is a
	non-singular $M$-matrix and
$
		\rho\left(\mathcal{M}^{-1}\mathcal{N}\right)<1.
	$
	Thus, $C$ is an $H$-matrix. By
	Lemma~\ref{lem:M-matrix-properties}(i), $C$ is non-singular.
	
	The splitting
	$
		\langle C\rangle-2|B|
		=
		\mathcal{M}-\mathcal{N}
	$
	is an $M$-splitting, and
	\begin{equation*}
		\rho\left(\mathcal{M}^{-1}\mathcal{N}\right)
		=
		\rho\left(
		2\langle C\rangle^{-1}|B|
		\right)
		=
		\widehat{\eta}
		<1.
	\end{equation*}
	Furthermore, Lemma~\ref{lem:M-matrix-properties}(i) gives
$
		|C^{-1}|
		\leq
		\langle C\rangle^{-1}.
$
	Consequently,
	\begin{equation*}
		\mathcal{K}
		=
		2|C^{-1}B|
		\leq
		2|C^{-1}||B|
		\leq
		2\langle C\rangle^{-1}|B|
		=
		\mathcal{M}^{-1}\mathcal{N}.
	\end{equation*}
	Both matrices in the last comparison are non-negative. Hence, the
	monotonicity result in Lemma~\ref{lem:perron-frobenius} yields
	\begin{equation*}
		\eta\leq\widehat{\eta}<1.
	\end{equation*}
	
	If $\widehat{\eta}>0$ and
	$0\leq\theta<1/\widehat{\eta}$, then
	\begin{equation*}
		\chi_\theta
		\leq
		\widehat{\eta}
		\left(
		|1-\theta|+\theta\widehat{\eta}
		\right)
		<1.
	\end{equation*}
	If $\widehat{\eta}=0$, then $\eta=0$ and consequently
	$\chi_\theta=0$ for every $\theta\geq0$. Therefore,
	Theorem~\ref{thm:global-convergence} completes the proof.
\end{proof}

	A simpler sufficient condition can be obtained from strict row
	diagonal dominance, as stated below.
	
	\begin{corollary}\label{cor:diagonal-dominance}
		Let $C=A+B=(c_{ij})$ and $B=(b_{ij})$. Suppose that
		\begin{equation}\label{eq:diagonal-dominance}
			|c_{ii}|
			>
			\sum_{\substack{j=1\\j\neq i}}^n |c_{ij}|
			+
			2\sum_{j=1}^n |b_{ij}|,
			\qquad i=1,\ldots,n.
		\end{equation}
		Then GAVE~\eqref{eq:GAVE} has a unique solution, and RMNSI converges
		to it from every initial vector for each $\theta\in[0,1]$.
	\end{corollary}
	
\begin{proof}
	Condition~\eqref{eq:diagonal-dominance} is equivalent to
	\begin{equation*}
		|c_{ii}|-2|b_{ii}|
		>
		\sum_{\substack{j=1\\j\neq i}}^n
		\left(|c_{ij}|+2|b_{ij}|\right),
		\qquad i=1,\ldots,n.
	\end{equation*}
	Hence, $\langle C\rangle-2|B|$ is a strictly row diagonally dominant
	$Z$-matrix with positive diagonal entries. By
	Lemma~\ref{lem:M-matrix-properties}(iii), it is a non-singular
	$M$-matrix. The conclusion now follows from
	Theorem~\ref{thm:M-matrix-condition}.
\end{proof}

The preceding results describe the global behaviour of RMNSI and
guarantee convergence from an arbitrary initial vector.However, the global estimate replaces the signed matrix \(2C^{-1}B\) by its component-wise absolute value and considers all possible sign patterns simultaneously. It therefore ignores cancellation among
the matrix entries and can yield a conservative condition.

A sharper description is possible near a solution \(x^*\) whose
components are all nonzero. In this case, the sign pattern of \(x^*\)
remains unchanged throughout a sufficiently small neighbourhood of
\(x^*\). Consequently, the mapping \(x\mapsto\max\{0,x\}\) agrees with a
fixed linear mapping in that neighbourhood, and the nonlinear RMNSI
iteration reduces locally to an exact linear error recurrence determined
by the active components of \(x^*\). This gives a convergence condition based on the actual sign pattern rather than on a global absolute-value bound.

Here, local convergence means that there exists a neighbourhood
\(\mathcal{V}\) of \(x^*\) such that every RMNSI sequence with
\(x^{(0)}\in\mathcal{V}\) remains in this neighbourhood and converges
to \(x^*\). The following theorem characterizes this behaviour through
the spectral radius of the corresponding local iteration matrix.

\begin{theorem}\label{thm:local-convergence}
	Let $\theta\geq0$, let $x^*$ be a solution of
	GAVE~\eqref{eq:GAVE}, and suppose that $C=A+B$ is non-singular.
	Assume that $x_i^*\neq0$ for every $i=1,\ldots,n$. Define
	\begin{equation}\label{eq:D-star}
		D_*
		=
		\operatorname{diag}(d_1,\ldots,d_n),
		\qquad
		d_i
		=
		\begin{cases}
			1, & x_i^*>0,\\
			0, & x_i^*<0,
		\end{cases}
	\end{equation}
	and
	\begin{equation}\label{eq:local-matrices}
		\mathcal{J}_*
		=
		2C^{-1}BD_*,
		\qquad
		\mathcal{R}_\theta
		=
		\mathcal{J}_*
		\left(
		(1-\theta)I_n+\theta\mathcal{J}_*
		\right).
	\end{equation}
	If
	\begin{equation}\label{eq:local-condition}
		\chi_{\mathrm{loc}}(\theta)
		:=
		\rho(\mathcal{R}_\theta)
		<1,
	\end{equation}
	then there exist an open neighbourhood $\mathcal{B}$ of $x^*$,
	a vector norm $\|\cdot\|_*$, and a constant $q\in(0,1)$ such that,
	for every $x^{(0)}\in\mathcal{B}$, the RMNSI iterates remain in
	$\mathcal{B}$ and satisfy
	\begin{equation}\label{eq:local-geometric-bound}
		\|x^{(k)}-x^*\|_*
		\leq
		q^k\|x^{(0)}-x^*\|_*,
		\qquad k=0,1,2,\ldots.
	\end{equation}
	Moreover,
	\begin{equation}\label{eq:local-factor}
		\chi_{\mathrm{loc}}(\theta)
		=
		\max_{\lambda\in\sigma(\mathcal{J}_*)}
		\left|
		\lambda\left(1-\theta+\theta\lambda\right)
		\right|.
	\end{equation}
\end{theorem}

\begin{proof}
	Since every component of $x^*$ is nonzero, there exists a
	neighbourhood $\mathcal{U}$ of $x^*$ in which the sign pattern
	remains unchanged. Therefore,
	\begin{equation}\label{eq:local-max-relation}
		\max\{0,x\}-\max\{0,x^*\}
		=
		D_*(x-x^*),
		\qquad x\in\mathcal{U}.
	\end{equation}
	Define:
$
		G_\theta(x)
		=
		(1-\theta)x+\theta T(x).
$
	Since $G_\theta$ is continuous and $G_\theta(x^*)=x^*$, there
	exists a neighbourhood $\mathcal{V}\subseteq\mathcal{U}$ of $x^*$
	such that
	$
		G_\theta(x)\in\mathcal{U},
		~ x\in\mathcal{V}.
$
	
	Let $e^{(k)}=x^{(k)}-x^*$. For $x^{(k)}\in\mathcal{V}$,
	the first relation in \eqref{eq:RMNSI} and
	\eqref{eq:local-max-relation} give
	$
		u^{(k)}-x^*
		=
		\mathcal{J}_*e^{(k)}.
$
	The second relation in \eqref{eq:RMNSI} then gives
$
		y^{(k)}-x^*
		=
		\left(
		(1-\theta)I_n+\theta\mathcal{J}_*
		\right)e^{(k)}.
$
	Since $y^{(k)}=G_\theta(x^{(k)})\in\mathcal{U}$, applying
	\eqref{eq:local-max-relation} to $y^{(k)}$ and using the third
	relation in \eqref{eq:RMNSI} yield
	\begin{equation}\label{eq:local-error-recurrence}
		e^{(k+1)}
		=
		\mathcal{R}_\theta e^{(k)}.
	\end{equation}
	
Since $\rho(\mathcal{R}_\theta)<1$,
Lemma~\ref{lem:spectral-properties}(iii) ensures that there exists a
vector norm $\|\cdot\|_*$ whose induced matrix norm satisfies
$
\|\mathcal{R}_\theta\|_*<1.
$
Choose any $q\in(0,1)$ such that
$
\|\mathcal{R}_\theta\|_*\leq q<1.
$

	Choose $\delta>0$ sufficiently small so that
	$
		\mathcal{B}
		=
		\left\{
		x\in\mathbb{R}^n:
		\|x-x^*\|_*<\delta
		\right\}
		\subseteq\mathcal{V}.
$
	If $x^{(k)}\in\mathcal{B}$, then
	\eqref{eq:local-error-recurrence} gives
	$
		\|e^{(k+1)}\|_*
		\leq
		q\|e^{(k)}\|_*
		<
		\delta.
	$
	
	Thus, $\mathcal{B}$ is invariant under the RMNSI iteration.
	By induction,
	$
		\|e^{(k)}\|_*
		\leq
		q^k\|e^{(0)}\|_*,
$
	which proves \eqref{eq:local-geometric-bound} and shows that
	$x^{(k)}\to x^*$.
	
	Finally,
	\begin{equation*}
		\mathcal{R}_\theta
		=
		(1-\theta)\mathcal{J}_*
		+\theta\mathcal{J}_*^2.
	\end{equation*}
	Applying the polynomial spectral-mapping result in
	Lemma~\ref{lem:spectral-properties}(i) gives
$
		\rho(\mathcal{R}_\theta)
		=
		\max_{\lambda\in\sigma(\mathcal{J}_*)}
		\left|
		\lambda\left(1-\theta+\theta\lambda\right)
		\right|,
	$
	which proves \eqref{eq:local-factor}.
\end{proof}

\begin{remark}\label{rem:local-parameter}
	Formula~\eqref{eq:local-factor} shows that a locally preferred
	relaxation parameter can be chosen from
	\begin{equation}\label{eq:local-optimal-theta}
		\theta_{\mathrm{loc}}^*
		\in
		\operatorname*{arg\,min}_{\theta\geq0}
		\max_{\lambda\in\sigma(\mathcal{J}_*)}
		\left|
		\lambda\left(1-\theta+\theta\lambda\right)
		\right|.
	\end{equation}
	A single closed formula for $\theta_{\mathrm{loc}}^*$ is not available
	in general because the minimizing value depends on the location of the
	eigenvalues of $\mathcal{J}_*$. In particular, the preferred local
	value can be larger than one, although $\theta=1$ minimizes the global
	bound in Theorem~\ref{thm:convergence-rate}.
	
	When $\theta=1$, the local iteration matrix reduces to
	\begin{equation*}
		\mathcal{R}_1=\mathcal{J}_*^2,
		\qquad
		\chi_{\mathrm{loc}}(1)=\rho(\mathcal{J}_*)^2.
	\end{equation*}
If \(\mathcal{R}_\theta\) is nilpotent, the local error becomes zero
after finitely many iterations once the iterates enter the
neighbourhood in which the active sign pattern remains fixed. This
observation can explain the small iteration counts in some numerical
examples, but it does not imply quadratic convergence in general.
\end{remark}

Table~\ref{tab:convergence-conditions} compares the global and local
convergence factors with the observed behaviour of RMNSI. In the last
column, $\mathrm{G}$ indicates that the global convergence condition is
satisfied, while $\mathrm{L}$ indicates that only the local condition
holds. The notation $\mathrm{NC}$ means that neither condition provides a
convergence guarantee. Since the established conditions are sufficient
but not necessary, $\mathrm{NC}$ does not imply divergence of the
computed sequence.

\begin{table}[width=\linewidth,cols=7,pos=H]
	\caption{\(\mathrm{Global~and~local~convergence~factors~for~the~
			parameter~choices~used~in~the~experiments~with~m=10.}\)}
	\label{tab:convergence-conditions}
	\footnotesize
	\setlength{\tabcolsep}{3pt}
	\renewcommand{\arraystretch}{1.05}
	\begin{tabular*}{\tblwidth}
		{@{\extracolsep{\fill}}ccccccc@{}}
		\toprule
		\(\mathrm{Example}\) &
		\(\mu\) &
		\(\theta\) &
		\(\eta\) &
		\(\chi_{\theta}\) &
		\(\chi_{\mathrm{loc}}(\theta)\) &
		\(\mathrm{Condition}\) \\
		\midrule
		
		\multirow{3}{*}{\(\ref{Ex1}\)}
		& \(0\) & \(0.6\) & \(6.1833\) & \(25.4136\) & \(0\) & \(\mathrm{L}\) \\
		& \(2\) & \(1.3\) & \(1.0797\) & \(1.8394\)  & \(0\) & \(\mathrm{L}\) \\
		& \(4\) & \(1.5\) & \(0.9723\) & \(1.9043\)  & \(0\) & \(\mathrm{L}\) \\
		\addlinespace
		
		\multirow{3}{*}{\(\ref{Ex2}\)}
		& \(0\) & \(1\)   & \(1.8006\) & \(3.2420\) & \(0\) & \(\mathrm{L}\) \\
		& \(2\) & \(1.3\) & \(1.0064\) & \(1.6187\) & \(0\) & \(\mathrm{L}\) \\
		& \(4\) & \(1.4\) & \(0.9509\) & \(1.6464\) & \(0\) & \(\mathrm{L}\) \\
		\addlinespace
		
		\(\ref{Ex3}\)
		& \(2\) & \(2.3\) & \(0.8924\) & \(2.9919\) & \(0.1556\) & \(\mathrm{L}\) \\
		\addlinespace
		
		\multirow{3}{*}{\(\ref{Ex4}\)}
		& \(0\) & \(0.7\) & \(1.4494\) & \(1.9053\) & \(1.0357\) & \(\mathrm{NC}\) \\
		& \(2\) & \(0.7\) & \(1.0496\) & \(1.0860\) & \(0.4563\) & \(\mathrm{L}\) \\
		& \(4\) & \(0.7\) & \(0.8497\) & \(0.7603\) & \(0.2505\) & \(\mathrm{G}\) \\
		\addlinespace
		
		\multirow{3}{*}{\(\ref{Ex5}\)}
		& \(0\) & \(0.7\) & \(1.4494\) & \(1.9053\) & \(1.0357\) & \(\mathrm{NC}\) \\
		& \(2\) & \(0.7\) & \(1.0496\) & \(1.0860\) & \(0.4563\) & \(\mathrm{L}\) \\
		& \(4\) & \(0.7\) & \(0.8497\) & \(0.7603\) & \(0.2505\) & \(\mathrm{G}\) \\
		\addlinespace
		
		\multirow{3}{*}{\(\ref{Ex6}\)}
		& \(0\) & \(1.6\) & \(0.4300\) & \(0.5538\) & \(0.0199\) & \(\mathrm{G}\) \\
		& \(2\) & \(1.6\) & \(0.4082\) & \(0.5115\) & \(0.0344\) & \(\mathrm{G}\) \\
		& \(4\) & \(1.5\) & \(0.3735\) & \(0.3960\) & \(0.0262\) & \(\mathrm{G}\) \\
		
		\bottomrule
	\end{tabular*}
\end{table}

For Examples~\ref{Ex1} and \ref{Ex2}, the global condition is not
satisfied, whereas $\chi_{\mathrm{loc}}(\theta)=0$ because their exact
solutions are strictly negative. Thus, after the iterates enter the
corresponding orthant, the maximum-based map becomes locally constant.
Example~\ref{Ex3} is also supported by the local condition despite the
failure of the global bound. For Examples~\ref{Ex4} and \ref{Ex5},
increasing $\mu$ improves both convergence factors, and the global
condition is satisfied when $\mu=4$. In the unshifted case, neither
condition covers the selected parameter, showing that convergence from
the prescribed initial vector can occur outside these sufficient
conditions. Finally, both conditions hold for Example~\ref{Ex6}, in
agreement with its stable numerical performance.

\section{Numerical experiments}\label{sec5}

This section evaluates the numerical performance of the proposed RMNSI method and compares it with the FPM-M method \cite{WangMu2025}, the M3 method \cite{WuHanLi2024}, the generalized Gauss-Seidel (GGS) method \cite{LuoEtAl2026}, and the shift-splitting fixed-point iteration (FPI-SS) method \cite{LiLiDou2023}. The test set covers problems with different structures.
Examples~\ref{Ex1}, \ref{Ex2}, \ref{Ex4}, and \ref{Ex5} are derived
from linear complementarity problems, Example~\ref{Ex3} is a dense
GAVE with a mixed-sign exact solution, and Example~\ref{Ex6} arises
from the first-order optimality condition of an asymmetric
ridge-regression model. For Example~\ref{Ex3}, the dense Hilbert matrix
is used in matrix-free form, since its explicit storage at $n=50000$
would require approximately $20$~GB of memory. Thus, the experiments
cover sparse and dense coefficient matrices, different solution sign
patterns, and an optimization-based application.

All methods are initialized with
$x^{(0)}=\boldsymbol{e}\in\mathbb{R}^n$, where
\(\boldsymbol{e}\in\mathbb{R}^n\) denotes the vector of all ones.
For FPI-SS, which also requires an auxiliary vector, we set
$y^{(0)}=|x^{(0)}|=\boldsymbol{e}$.
The iterations are terminated as soon as
$\mathrm{RES}\leq 10^{-6}$ or when the maximum number of $500$ iterations is reached. For an approximation $x^{(k)}$ to the exact solution $x^*$ of GAVE~\eqref{eq:GAVE}, the reported residual and error are defined by
\begin{equation}\label{eq:numerical-measures}
	\mathrm{RES}^{(k)}
	=
	\frac{
		\left\|Ax^{(k)}-B|x^{(k)}|-b\right\|_2
	}{
		\|b\|_2
	},
	\qquad
	\mathrm{ERR}^{(k)}
	=
	\left\|x^{(k)}-x^*\right\|_2.
\end{equation}
Here, $\mathrm{IT}$ denotes the number of iterations required to satisfy the stopping criterion, whereas $\mathrm{CPU}$ denotes the elapsed computing time in seconds measured using the MATLAB commands \texttt{tic} and \texttt{toc}. The cost of the preliminary parameter searches is not included in the reported CPU times. The notation ``Fail'' indicates that the prescribed tolerance was not attained within $500$ iterations.

The parameters of the comparison methods are selected according to the recommendations in their respective source papers whenever applicable. The parameters of FPI-SS and RMNSI are determined from preliminary experiments at the smallest tested problem size and are then kept fixed as the dimension increases. Their values are displayed beneath the corresponding method names in each table. This strategy avoids selecting parameters separately for every problem dimension and also helps assess the sensitivity of the methods to increasing problem size.

All computations were performed in MATLAB R2025b on a system equipped with a 13th Gen Intel(R) Core(TM) i7-13700HX processor and $16$ GB of RAM, running a 64-bit Windows operating system. The numerical results are reported in Tables~\ref{tab:ex1-comparison}--\ref{tab:ex6-comparison}.

To describe the coefficient matrices of the test problems consistently,
we first introduce some notation.For \(\alpha,\beta,\gamma\in\mathbb{R}\), define
\begin{equation*}
	\mathcal{U}_m(\alpha,\beta,\gamma)
	:=
	\operatorname{tridiag}(\alpha,\beta,\gamma)
	=
	\begin{pmatrix}
		\beta  & \gamma &         &         &        \\
		\alpha & \beta  & \gamma  &         &        \\
		& \ddots & \ddots  & \ddots  &        \\
		&        & \alpha  & \beta   & \gamma \\
		&        &         & \alpha  & \beta
	\end{pmatrix}
	\in\mathbb{R}^{m\times m}.
\end{equation*}
For \(\delta,\varepsilon,\mu\in\mathbb{R}\), define
\begin{equation*}
	\mathcal{V}_{m,\mu}
	(\alpha,\beta,\gamma;\delta,\varepsilon)
	:=
	I_m\otimes
	\mathcal{U}_m(\alpha,\beta,\gamma)
	+
	\mathcal{U}_m(\delta,\mu,\varepsilon)
	\otimes I_m.
\end{equation*}
Here, \(I_r\) denotes the identity matrix of order \(r\), and
\(\otimes\) denotes the Kronecker product. Equivalently, setting
\(U=\mathcal{U}_m(\alpha,\beta,\gamma)\), we have
\begin{equation*}
	\mathcal{V}_{m,\mu}
	(\alpha,\beta,\gamma;\delta,\varepsilon)
	=
	\begin{pmatrix}
		U+\mu I_m        & \varepsilon I_m &                 &                 & \\
		\delta I_m       & U+\mu I_m       & \varepsilon I_m &                 & \\
		& \ddots          & \ddots          & \ddots          & \\
		&                 & \delta I_m       & U+\mu I_m       & \varepsilon I_m \\
		&                 &                  & \delta I_m       & U+\mu I_m
	\end{pmatrix}
	\in\mathbb{R}^{n\times n},
	\qquad n=m^2.
\end{equation*}

Several test problems considered below are derived from the linear
complementarity problem
\begin{equation*}
	z\geq0,\qquad
	w=Vz+q\geq0,\qquad
	z^{\mathsf T}w=0.
\end{equation*}
We employ two GAVE reformulations of this problem. For the first one,
introduce
$
z=|x|-x,
~
w=|x|+x.
$
Substitution into \(w=Vz+q\) gives
$
|x|+x=V(|x|-x)+q,
$
which is equivalent to GAVE~\eqref{eq:GAVE} with
\begin{equation}\label{eq:LCP-GAVE-1}
	A=V+I_n,\qquad
	B=V-I_n,\qquad
	b=q.
\end{equation}
If \(z^*\geq0\), \(q=-Vz^*\), and \(w^*=0\), then
\((z^*,w^*)\) satisfies the complementarity conditions, and the
corresponding solution of GAVE~\eqref{eq:GAVE} is
$
x^*=-\frac{1}{2}z^*.
$

For the second reformulation, let \(D\) be a positive diagonal matrix
and introduce the scaled transformation
$
z=|x|+x,
~
w=D(|x|-x).
$
Substitution into \(w=Vz+q\) yields
$
D(|x|-x)=V(|x|+x)+q,
$
and hence GAVE~\eqref{eq:GAVE} is obtained with
\begin{equation}\label{eq:LCP-GAVE-2}
	A=V+D,\qquad
	B=D-V,\qquad
	b=-q.
\end{equation}
In this formulation, if \(z^*\geq0\), \(q=-Vz^*\), and \(w^*=0\),
then the associated exact solution is
$
x^*=\frac{1}{2}z^*.
$

\begin{example}\label{Ex1} (\cite{WuHanLi2024}). Let
	$V=\mathcal{V}_{m,\mu}(-1,4,-1;-1,-1)$ and define
$
		z_i^*=
		\begin{cases}
			1, & i\text{ is odd},\\
			2, & i\text{ is even}.
		\end{cases}
$
	Thus, $z^*=(1,2,1,2,\ldots)^{\mathsf T}$. Setting $q=-Vz^*$ gives
	$w^*=Vz^*+q=0$, so $(z^*,w^*)$ satisfies the complementarity
	conditions.
	
	Using \eqref{eq:LCP-GAVE-1}, the matrices and right-hand side of
	GAVE~\eqref{eq:GAVE} are
	$A=V+I_n$, $B=V-I_n$, and $b=q$. The corresponding exact solution is
$
		x^*=-\frac{1}{2}z^*
		=
		\left(
		-\frac{1}{2},-1,-\frac{1}{2},-1,\ldots
		\right)^{\mathsf T}.
$
\end{example}

\begin{table}[width=\linewidth,cols=13,pos=H]
	\caption{\(\mathrm{Performance~comparison~of~tested~methods~
			for~Example~\ref{Ex1}.}\)}
	\label{tab:ex1-comparison}
	\footnotesize
	\setlength{\tabcolsep}{1.2pt}
	\renewcommand{\arraystretch}{0.9}
	
	\begin{tabular*}{\tblwidth}
		{@{\extracolsep{\fill}}l*{12}{c}@{}}
		\toprule
		
		\(\mathrm{FPM\mbox{-}M}\)
		& \multicolumn{4}{c}{\(\mu=0\)}
		& \multicolumn{4}{c}{\(\mu=2\)}
		& \multicolumn{4}{c}{\(\mu=4\)} \\
		\cmidrule(lr){2-5}
		\cmidrule(lr){6-9}
		\cmidrule(lr){10-13}
		\(m\)
		& \(\mathrm{IT}\) & \(\mathrm{CPU}\) & \(\mathrm{RES}\) & \(\mathrm{ERR}\)
		& \(\mathrm{IT}\) & \(\mathrm{CPU}\) & \(\mathrm{RES}\) & \(\mathrm{ERR}\)
		& \(\mathrm{IT}\) & \(\mathrm{CPU}\) & \(\mathrm{RES}\) & \(\mathrm{ERR}\) \\
		\midrule
		\(10\)
		& \(2\) & \(0.0013\) & \(2.4944\mathrm{e}{-16}\) & \(1.4676\mathrm{e}{-15}\)
		& \(3\) & \(0.0001\) & \(2.8964\mathrm{e}{-16}\) & \(1.4916\mathrm{e}{-15}\)
		& \(3\) & \(0.0001\) & \(3.4991\mathrm{e}{-16}\) & \(1.6690\mathrm{e}{-15}\) \\
		\(50\)
		& \(2\) & \(0.0014\) & \(2.9339\mathrm{e}{-16}\) & \(1.5439\mathrm{e}{-14}\)
		& \(3\) & \(0.001\) & \(3.3284\mathrm{e}{-16}\) & \(7.0724\mathrm{e}{-15}\)
		& \(3\) & \(0.0011\) & \(8.5596\mathrm{e}{-17}\) & \(1.4272\mathrm{e}{-15}\) \\
		\(100\)
		& \(2\) & \(0.0089\) & \(2.7231\mathrm{e}{-16}\) & \(8.0563\mathrm{e}{-14}\)
		& \(3\) & \(0.0059\) & \(3.3851\mathrm{e}{-16}\) & \(1.3792\mathrm{e}{-14}\)
		& \(3\) & \(0.0042\) & \(8.6295\mathrm{e}{-17}\) & \(2.8668\mathrm{e}{-15}\) \\
		\(150\)
		& \(2\) & \(0.0083\) & \(2.7895\mathrm{e}{-16}\) & \(2.0289\mathrm{e}{-13}\)
		& \(3\) & \(0.0126\) & \(3.3394\mathrm{e}{-16}\) & \(2.0881\mathrm{e}{-14}\)
		& \(3\) & \(0.0102\) & \(9.2365\mathrm{e}{-17}\) & \(4.3554\mathrm{e}{-15}\) \\
		\(200\)
		& \(2\) & \(0.0184\) & \(2.9202\mathrm{e}{-16}\) & \(3.8314\mathrm{e}{-13}\)
		& \(3\) & \(0.0268\) & \(3.3638\mathrm{e}{-16}\) & \(2.6930\mathrm{e}{-14}\)
		& \(3\) & \(0.0256\) & \(9.4219\mathrm{e}{-17}\) & \(5.8178\mathrm{e}{-15}\) \\
		\(250\)
		& \(2\) & \(0.0298\) & \(2.7684\mathrm{e}{-16}\) & \(4.8669\mathrm{e}{-13}\)
		& \(3\) & \(0.1101\) & \(3.3123\mathrm{e}{-16}\) & \(3.3335\mathrm{e}{-14}\)
		& \(3\) & \(0.1007\) & \(9.0273\mathrm{e}{-17}\) & \(6.8355\mathrm{e}{-15}\) \\
		
		\midrule
		
		\(\mathrm{M3}\)
		& \multicolumn{4}{c}{\(\mu=0\)}
		& \multicolumn{4}{c}{\(\mu=2\)}
		& \multicolumn{4}{c}{\(\mu=4\)} \\
		\cmidrule(lr){2-5}
		\cmidrule(lr){6-9}
		\cmidrule(lr){10-13}
		\(m\)
		& \(\mathrm{IT}\) & \(\mathrm{CPU}\) & \(\mathrm{RES}\) & \(\mathrm{ERR}\)
		& \(\mathrm{IT}\) & \(\mathrm{CPU}\) & \(\mathrm{RES}\) & \(\mathrm{ERR}\)
		& \(\mathrm{IT}\) & \(\mathrm{CPU}\) & \(\mathrm{RES}\) & \(\mathrm{ERR}\) \\
		\midrule
		\(10\)
		& \(11\) & \(0.0331\) & \(5.1787\mathrm{e}{-16}\) & \(1.1748\mathrm{e}{-14}\)
		& \(2\) & \(0.0002\) & \(4.6765\mathrm{e}{-16}\) & \(2.4248\mathrm{e}{-15}\)
		& \(2\) & \(0.0002\) & \(2.5388\mathrm{e}{-16}\) & \(1.1551\mathrm{e}{-15}\) \\
		\(50\)
		& \(13\) & \(0.0101\) & \(6.8883\mathrm{e}{-16}\) & \(6.7175\mathrm{e}{-13}\)
		& \(2\) & \(0.0029\) & \(5.8913\mathrm{e}{-16}\) & \(1.9171\mathrm{e}{-14}\)
		& \(2\) & \(0.003\) & \(3.6640\mathrm{e}{-16}\) & \(8.0789\mathrm{e}{-15}\) \\
		\(100\)
		& \(13\) & \(0.0813\) & \(5.7980\mathrm{e}{-16}\) & \(1.1729\mathrm{e}{-11}\)
		& \(2\) & \(0.018\) & \(5.0988\mathrm{e}{-16}\) & \(2.5016\mathrm{e}{-14}\)
		& \(2\) & \(0.0185\) & \(3.4029\mathrm{e}{-16}\) & \(1.4450\mathrm{e}{-14}\) \\
		\(150\)
		& \(13\) & \(0.1543\) & \(5.8840\mathrm{e}{-16}\) & \(3.8675\mathrm{e}{-11}\)
		& \(2\) & \(0.0464\) & \(5.2912\mathrm{e}{-16}\) & \(3.8670\mathrm{e}{-14}\)
		& \(2\) & \(0.0428\) & \(3.4352\mathrm{e}{-16}\) & \(2.1834\mathrm{e}{-14}\) \\
		\(200\)
		& \(13\) & \(0.2402\) & \(5.9701\mathrm{e}{-16}\) & \(9.1592\mathrm{e}{-11}\)
		& \(2\) & \(0.094\) & \(5.3154\mathrm{e}{-16}\) & \(4.9273\mathrm{e}{-14}\)
		& \(2\) & \(0.1057\) & \(3.5484\mathrm{e}{-16}\) & \(3.0020\mathrm{e}{-14}\) \\
		\(250\)
		& \(13\) & \(0.4136\) & \(6.0413\mathrm{e}{-16}\) & \(1.7899\mathrm{e}{-10}\)
		& \(2\) & \(0.125\) & \(5.3956\mathrm{e}{-16}\) & \(6.2389\mathrm{e}{-14}\)
		& \(2\) & \(0.1471\) & \(3.5085\mathrm{e}{-16}\) & \(3.7212\mathrm{e}{-14}\) \\
		
		\midrule
		
		\(\mathrm{GGS}\)
		& \multicolumn{4}{c}{\(\mu=0\)}
		& \multicolumn{4}{c}{\(\mu=2\)}
		& \multicolumn{4}{c}{\(\mu=4\)} \\
		\cmidrule(lr){2-5}
		\cmidrule(lr){6-9}
		\cmidrule(lr){10-13}
		\(m\)
		& \(\mathrm{IT}\) & \(\mathrm{CPU}\) & \(\mathrm{RES}\) & \(\mathrm{ERR}\)
		& \(\mathrm{IT}\) & \(\mathrm{CPU}\) & \(\mathrm{RES}\) & \(\mathrm{ERR}\)
		& \(\mathrm{IT}\) & \(\mathrm{CPU}\) & \(\mathrm{RES}\) & \(\mathrm{ERR}\) \\
		\midrule
		\(10\)
		& \(137\) & \(0.0808\) & \(9.2855\mathrm{e}{-07}\) & \(6.2777\mathrm{e}{-05}\)
		& \(18\) & \(0.0101\) & \(5.6951\mathrm{e}{-07}\) & \(5.7401\mathrm{e}{-06}\)
		& \(12\) & \(0.006\) & \(5.8836\mathrm{e}{-07}\) & \(5.2331\mathrm{e}{-06}\) \\
		\(50\)
		& \(\mathrm{Fail}\) & \(--\) & \(--\) & \(--\)
		& \(20\) & \(0.744\) & \(5.4837\mathrm{e}{-07}\) & \(2.9485\mathrm{e}{-05}\)
		& \(13\) & \(0.4603\) & \(4.4917\mathrm{e}{-07}\) & \(2.0404\mathrm{e}{-05}\) \\
		\(100\)
		& \(\mathrm{Fail}\) & \(--\) & \(--\) & \(--\)
		& \(20\) & \(10.964\) & \(6.1059\mathrm{e}{-07}\) & \(6.5252\mathrm{e}{-05}\)
		& \(13\) & \(6.695\) & \(4.8526\mathrm{e}{-07}\) & \(4.3927\mathrm{e}{-05}\) \\
		\(150\)
		& \(\mathrm{Fail}\) & \(--\) & \(--\) & \(--\)
		& \(20\) & \(62.191\) & \(6.3167\mathrm{e}{-07}\) & \(1.0102\mathrm{e}{-04}\)
		& \(13\) & \(41.768\) & \(4.9741\mathrm{e}{-07}\) & \(6.7449\mathrm{e}{-05}\) \\
		\(200\)
		& \(\mathrm{Fail}\) & \(--\) & \(--\) & \(--\)
		& \(20\) & \(520.47\) & \(6.4227\mathrm{e}{-07}\) & \(1.3678\mathrm{e}{-04}\)
		& \(13\) & \(437.18\) & \(5.0350\mathrm{e}{-07}\) & \(9.0970\mathrm{e}{-05}\) \\
		\(250\)
		& \(\mathrm{Fail}\) & \(--\) & \(--\) & \(--\)
		& \(20\) & \(1018.6\) & \(6.4865\mathrm{e}{-07}\) & \(1.7254\mathrm{e}{-04}\)
		& \(13\) & \(587.11\) & \(5.0717\mathrm{e}{-07}\) & \(1.1449\mathrm{e}{-04}\) \\
		
		\midrule
		
		\(\mathrm{FPI\mbox{-}SS}\)
		& \multicolumn{4}{c}{
			\makecell[c]{\(\mu=0\)\\
				\((\alpha=2.3,\omega=0.35)\)}}
		& \multicolumn{4}{c}{
			\makecell[c]{\(\mu=2\)\\
				\((\alpha=12.1,\omega=0.8)\)}}
		& \multicolumn{4}{c}{
			\makecell[c]{\(\mu=4\)\\
				\((\alpha=19.8,\omega=0.9)\)}} \\
		\cmidrule(lr){2-5}
		\cmidrule(lr){6-9}
		\cmidrule(lr){10-13}
		\(m\)
		& \(\mathrm{IT}\) & \(\mathrm{CPU}\) & \(\mathrm{RES}\) & \(\mathrm{ERR}\)
		& \(\mathrm{IT}\) & \(\mathrm{CPU}\) & \(\mathrm{RES}\) & \(\mathrm{ERR}\)
		& \(\mathrm{IT}\) & \(\mathrm{CPU}\) & \(\mathrm{RES}\) & \(\mathrm{ERR}\) \\
		\midrule
		\(10\)
		& \(72\) & \(0.0028\) & \(9.5823\mathrm{e}{-07}\) & \(6.7062\mathrm{e}{-05}\)
		& \(13\) & \(0.0002\) & \(9.6205\mathrm{e}{-07}\) & \(8.2687\mathrm{e}{-06}\)
		& \(11\) & \(0.0001\) & \(8.0047\mathrm{e}{-07}\) & \(6.4454\mathrm{e}{-06}\) \\
		\(50\)
		& \(\mathrm{Fail}\) & \(--\) & \(--\) & \(--\)
		& \(17\) & \(0.005\) & \(7.3254\mathrm{e}{-07}\) & \(3.9646\mathrm{e}{-05}\)
		& \(13\) & \(0.0032\) & \(3.2302\mathrm{e}{-07}\) & \(1.4706\mathrm{e}{-05}\) \\
		\(100\)
		& \(\mathrm{Fail}\) & \(--\) & \(--\) & \(--\)
		& \(17\) & \(0.0206\) & \(8.7202\mathrm{e}{-07}\) & \(9.3480\mathrm{e}{-05}\)
		& \(13\) & \(0.014\) & \(3.8196\mathrm{e}{-07}\) & \(3.4611\mathrm{e}{-05}\) \\
		\(150\)
		& \(\mathrm{Fail}\) & \(--\) & \(--\) & \(--\)
		& \(17\) & \(0.0644\) & \(9.1901\mathrm{e}{-07}\) & \(1.4727\mathrm{e}{-04}\)
		& \(13\) & \(0.0549\) & \(4.0183\mathrm{e}{-07}\) & \(5.4513\mathrm{e}{-05}\) \\
		\(200\)
		& \(\mathrm{Fail}\) & \(--\) & \(--\) & \(--\)
		& \(17\) & \(0.385\) & \(9.4261\mathrm{e}{-07}\) & \(2.0104\mathrm{e}{-04}\)
		& \(13\) & \(0.2717\) & \(4.1180\mathrm{e}{-07}\) & \(7.4416\mathrm{e}{-05}\) \\
		\(250\)
		& \(\mathrm{Fail}\) & \(--\) & \(--\) & \(--\)
		& \(17\) & \(0.2114\) & \(9.5681\mathrm{e}{-07}\) & \(2.5482\mathrm{e}{-04}\)
		& \(13\) & \(0.1735\) & \(4.1780\mathrm{e}{-07}\) & \(9.4318\mathrm{e}{-05}\) \\
		
		\midrule
		
		\(\mathbf{RMNSI}\)
		& \multicolumn{4}{c}{
			\makecell[c]{\(\mu=0\)\\ \((\theta=0.6)\)}}
		& \multicolumn{4}{c}{
			\makecell[c]{\(\mu=2\)\\ \((\theta=1.3)\)}}
		& \multicolumn{4}{c}{
			\makecell[c]{\(\mu=4\)\\ \((\theta=1.5)\)}} \\
		\cmidrule(lr){2-5}
		\cmidrule(lr){6-9}
		\cmidrule(lr){10-13}
		\(m\)
		& \(\mathrm{IT}\) & \(\mathrm{CPU}\) & \(\mathrm{RES}\) & \(\mathrm{ERR}\)
		& \(\mathrm{IT}\) & \(\mathrm{CPU}\) & \(\mathrm{RES}\) & \(\mathrm{ERR}\)
		& \(\mathrm{IT}\) & \(\mathrm{CPU}\) & \(\mathrm{RES}\) & \(\mathrm{ERR}\) \\
		\midrule
		\(10\)
		& \(1\) & \(0.001\) & \(2.4944\mathrm{e}{-16}\) & \(1.4676\mathrm{e}{-15}\)
		& \(1\) & \(<0.0001\) & \(2.8964\mathrm{e}{-16}\) & \(1.4916\mathrm{e}{-15}\)
		& \(1\) & \(<0.0001\) & \(3.4991\mathrm{e}{-16}\) & \(1.6690\mathrm{e}{-15}\) \\
		\(50\)
		& \(1\) & \(0.0011\) & \(2.9339\mathrm{e}{-16}\) & \(1.5439\mathrm{e}{-14}\)
		& \(1\) & \(0.0008\) & \(3.3284\mathrm{e}{-16}\) & \(7.0724\mathrm{e}{-15}\)
		& \(1\) & \(0.0007\) & \(8.5596\mathrm{e}{-17}\) & \(1.4272\mathrm{e}{-15}\) \\
		\(100\)
		& \(1\) & \(0.0042\) & \(2.7231\mathrm{e}{-16}\) & \(8.0563\mathrm{e}{-14}\)
		& \(1\) & \(0.0028\) & \(3.3851\mathrm{e}{-16}\) & \(1.3792\mathrm{e}{-14}\)
		& \(1\) & \(0.0034\) & \(8.6295\mathrm{e}{-17}\) & \(2.8668\mathrm{e}{-15}\) \\
		\(150\)
		& \(1\) & \(0.0081\) & \(2.7895\mathrm{e}{-16}\) & \(2.0289\mathrm{e}{-13}\)
		& \(1\) & \(0.0067\) & \(3.3394\mathrm{e}{-16}\) & \(2.0881\mathrm{e}{-14}\)
		& \(1\) & \(0.0066\) & \(9.2365\mathrm{e}{-17}\) & \(4.3554\mathrm{e}{-15}\) \\
		\(200\)
		& \(1\) & \(0.0183\) & \(2.9202\mathrm{e}{-16}\) & \(3.8314\mathrm{e}{-13}\)
		& \(1\) & \(0.043\) & \(3.3638\mathrm{e}{-16}\) & \(2.6930\mathrm{e}{-14}\)
		& \(1\) & \(0.0161\) & \(9.4219\mathrm{e}{-17}\) & \(5.8178\mathrm{e}{-15}\) \\
		\(250\)
		& \(1\) & \(0.0262\) & \(2.7684\mathrm{e}{-16}\) & \(4.8669\mathrm{e}{-13}\)
		& \(1\) & \(0.0301\) & \(3.3123\mathrm{e}{-16}\) & \(3.3335\mathrm{e}{-14}\)
		& \(1\) & \(0.03\) & \(9.0273\mathrm{e}{-17}\) & \(6.8355\mathrm{e}{-15}\) \\
		
		\bottomrule
	\end{tabular*}
\end{table}

Table~\ref{tab:ex1-comparison} shows that, for
Example~\ref{Ex1}, RMNSI performs consistently for all tested values of
$m$ and $\mu$. Its convergence behaviour is unaffected by the increasing
problem size, and the computed solutions are accurate to nearly machine
precision. FPM-M and M3 also produce accurate solutions but require more
computational work. When $\mu=0$, GGS and FPI-SS fail for the larger
systems, whereas RMNSI remains stable. For positive values of $\mu$, all
methods converge, but GGS becomes considerably more expensive as $m$
increases. Overall, RMNSI provides the best balance of efficiency,
accuracy, and stability for this example.
	
\begin{figure}[width=\linewidth,pos=H]
	\centering
	\includegraphics[width=\linewidth]{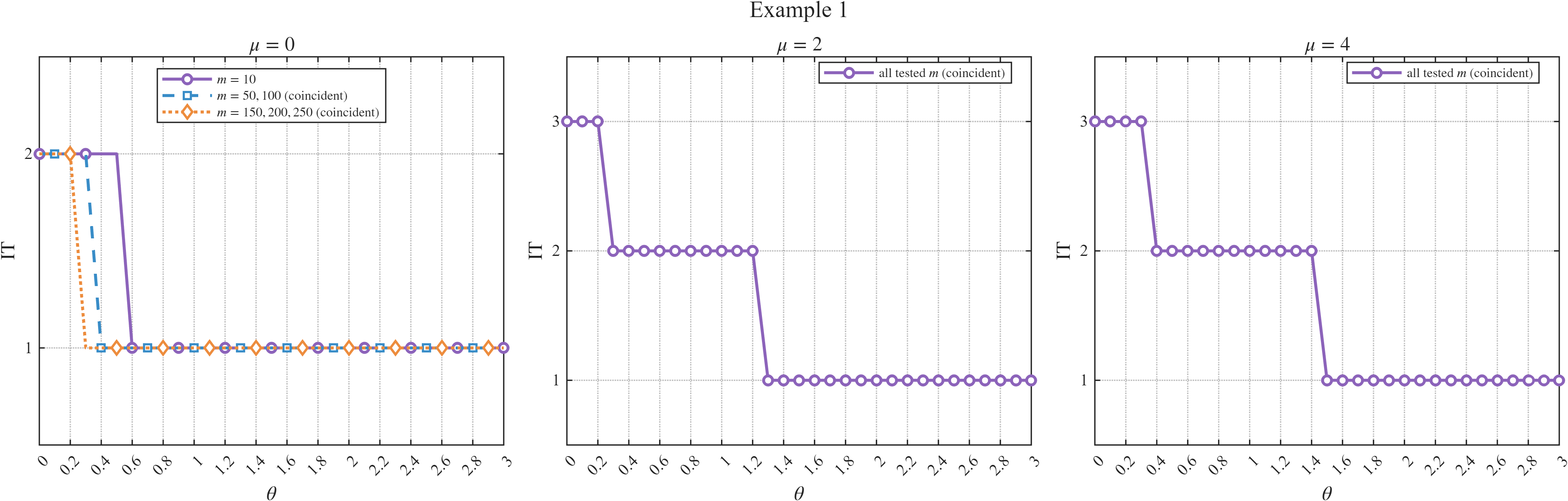}
	\caption{\(\mathrm{Effect~of~the~relaxation~parameter~\theta~on~the~iteration~
			count~of~RMNSI~for~Example~\ref{Ex1}~with~different~values~of~m.}\)}
	\label{fig:theta-ex1}
\end{figure}

Figure~\ref{fig:theta-ex1} illustrates the effect of $\theta$ on RMNSI
for Example~\ref{Ex1}. Positive values of $\mu$ move the best parameter
range towards larger values of $\theta$, while the coincident curves
indicate little dependence on $m$.

\begin{example}\label{Ex2} (\cite{WangMu2025}). Let
$
		V=
		\mathcal{V}_{m,\mu}
		\left(
		-\frac{1}{2},4,-\frac{3}{2};
		-\frac{1}{2},-\frac{3}{2}
		\right).
$
	The matrix $V$ is generally nonsymmetric because its lower and upper
	off-diagonal entries are different. Let $z^*=1.2\boldsymbol{e}$ and
	set $q=-Vz^*$. Then $w^*=0$, and hence $(z^*,w^*)$ satisfies the
	complementarity conditions.
	
	By \eqref{eq:LCP-GAVE-1}, GAVE~\eqref{eq:GAVE} is defined by
	$A=V+I_n$, $B=V-I_n$, and $b=q$. Its exact solution is
	$x^*=-z^*/2=-0.6\boldsymbol{e}$.
\end{example}

\begin{figure}[width=\linewidth,pos=H]
	\centering
	\includegraphics[width=\linewidth]{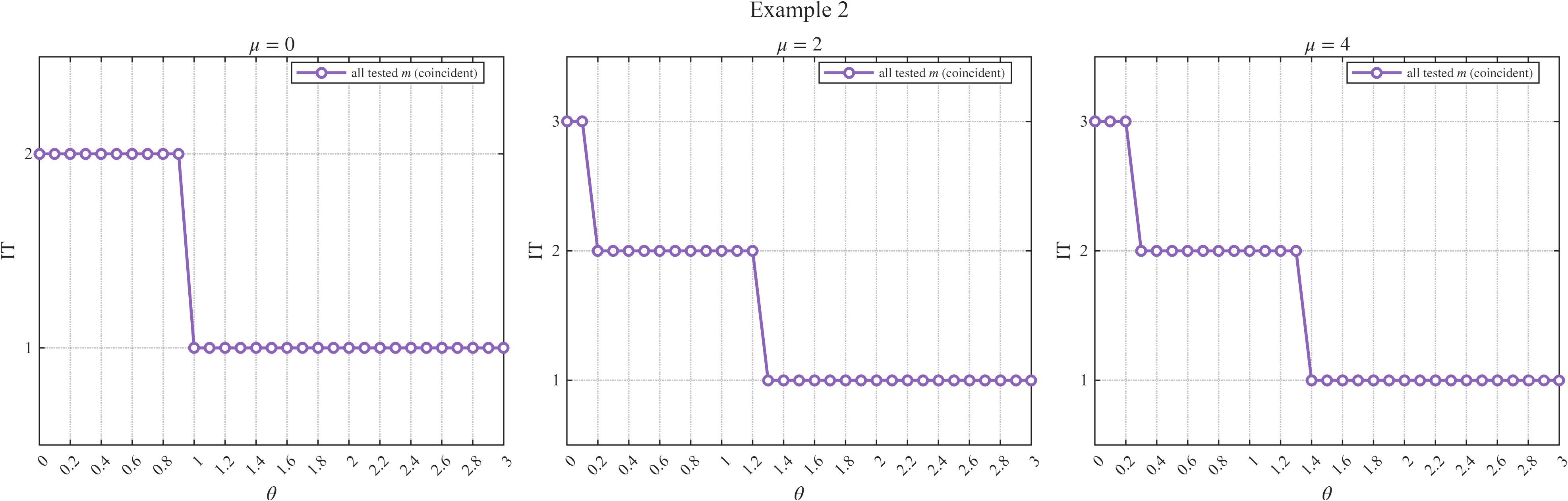}
	\caption{\(\mathrm{Effect~of~the~relaxation~parameter~\theta~on~the~iteration~
			count~of~RMNSI~for~Example~\ref{Ex2}~with~different~values~of~m.}\)}
	\label{fig:theta-ex2}
\end{figure}

\begin{table}[width=\linewidth,cols=13,pos=H]
	\caption{\(\mathrm{Performance~comparison~of~tested~methods~
			for~Example~\ref{Ex2}.}\)}
	\label{tab:ex2-comparison}
	\footnotesize
	\setlength{\tabcolsep}{1.2pt}
	\renewcommand{\arraystretch}{0.9}
	
	\begin{tabular*}{\tblwidth}
		{@{\extracolsep{\fill}}l*{12}{c}@{}}
		\toprule
		
		\(\mathrm{FPM\mbox{-}M}\)
		& \multicolumn{4}{c}{\(\mu=0\)}
		& \multicolumn{4}{c}{\(\mu=2\)}
		& \multicolumn{4}{c}{\(\mu=4\)} \\
		\cmidrule(lr){2-5}
		\cmidrule(lr){6-9}
		\cmidrule(lr){10-13}
		\(m\)
		& \(\mathrm{IT}\) & \(\mathrm{CPU}\) & \(\mathrm{RES}\) & \(\mathrm{ERR}\)
		& \(\mathrm{IT}\) & \(\mathrm{CPU}\) & \(\mathrm{RES}\) & \(\mathrm{ERR}\)
		& \(\mathrm{IT}\) & \(\mathrm{CPU}\) & \(\mathrm{RES}\) & \(\mathrm{ERR}\) \\
		\midrule
		\(10\)
		& \(2\) & \(0.0002\) & \(7.6833\mathrm{e}{-16}\) & \(2.0502\mathrm{e}{-15}\)
		& \(3\) & \(0.0007\) & \(3.3321\mathrm{e}{-16}\) & \(1.1213\mathrm{e}{-15}\)
		& \(3\) & \(<0.0001\) & \(2.0177\mathrm{e}{-16}\) & \(1.3461\mathrm{e}{-15}\) \\
		\(50\)
		& \(2\) & \(0.0008\) & \(1.0553\mathrm{e}{-15}\) & \(6.5390\mathrm{e}{-15}\)
		& \(3\) & \(0.001\) & \(3.1954\mathrm{e}{-16}\) & \(8.2769\mathrm{e}{-15}\)
		& \(3\) & \(0.0011\) & \(2.3034\mathrm{e}{-16}\) & \(8.4018\mathrm{e}{-15}\) \\
		\(100\)
		& \(2\) & \(0.0074\) & \(1.4490\mathrm{e}{-15}\) & \(2.1534\mathrm{e}{-14}\)
		& \(3\) & \(0.0045\) & \(3.3583\mathrm{e}{-16}\) & \(1.6881\mathrm{e}{-14}\)
		& \(3\) & \(0.005\) & \(2.3636\mathrm{e}{-16}\) & \(1.6980\mathrm{e}{-14}\) \\
		\(150\)
		& \(2\) & \(0.0365\) & \(1.8141\mathrm{e}{-15}\) & \(2.2149\mathrm{e}{-14}\)
		& \(3\) & \(0.018\) & \(3.3485\mathrm{e}{-16}\) & \(2.5504\mathrm{e}{-14}\)
		& \(3\) & \(0.01\) & \(2.4456\mathrm{e}{-16}\) & \(2.5775\mathrm{e}{-14}\) \\
		\(200\)
		& \(2\) & \(0.027\) & \(2.1225\mathrm{e}{-15}\) & \(3.0834\mathrm{e}{-14}\)
		& \(3\) & \(0.0239\) & \(3.3605\mathrm{e}{-16}\) & \(3.4051\mathrm{e}{-14}\)
		& \(3\) & \(0.0267\) & \(2.4044\mathrm{e}{-16}\) & \(3.4193\mathrm{e}{-14}\) \\
		\(250\)
		& \(2\) & \(0.0281\) & \(2.3221\mathrm{e}{-15}\) & \(4.8257\mathrm{e}{-14}\)
		& \(3\) & \(0.0405\) & \(3.3890\mathrm{e}{-16}\) & \(4.2272\mathrm{e}{-14}\)
		& \(3\) & \(0.0405\) & \(2.4719\mathrm{e}{-16}\) & \(4.3391\mathrm{e}{-14}\) \\
		
		\midrule
		
		\(\mathrm{M3}\)
		& \multicolumn{4}{c}{\(\mu=0\)}
		& \multicolumn{4}{c}{\(\mu=2\)}
		& \multicolumn{4}{c}{\(\mu=4\)} \\
		\cmidrule(lr){2-5}
		\cmidrule(lr){6-9}
		\cmidrule(lr){10-13}
		\(m\)
		& \(\mathrm{IT}\) & \(\mathrm{CPU}\) & \(\mathrm{RES}\) & \(\mathrm{ERR}\)
		& \(\mathrm{IT}\) & \(\mathrm{CPU}\) & \(\mathrm{RES}\) & \(\mathrm{ERR}\)
		& \(\mathrm{IT}\) & \(\mathrm{CPU}\) & \(\mathrm{RES}\) & \(\mathrm{ERR}\) \\
		\midrule
		\(10\)
		& \(2\) & \(0.0004\) & \(7.6833\mathrm{e}{-16}\) & \(2.0502\mathrm{e}{-15}\)
		& \(2\) & \(0.0004\) & \(3.3321\mathrm{e}{-16}\) & \(1.1213\mathrm{e}{-15}\)
		& \(2\) & \(0.0004\) & \(2.0177\mathrm{e}{-16}\) & \(1.3461\mathrm{e}{-15}\) \\
		\(50\)
		& \(2\) & \(0.0087\) & \(1.0553\mathrm{e}{-15}\) & \(6.5390\mathrm{e}{-15}\)
		& \(2\) & \(0.0109\) & \(3.1954\mathrm{e}{-16}\) & \(8.2769\mathrm{e}{-15}\)
		& \(2\) & \(0.009\) & \(2.3034\mathrm{e}{-16}\) & \(8.4018\mathrm{e}{-15}\) \\
		\(100\)
		& \(2\) & \(0.0398\) & \(1.4490\mathrm{e}{-15}\) & \(2.1534\mathrm{e}{-14}\)
		& \(2\) & \(0.04\) & \(3.3583\mathrm{e}{-16}\) & \(1.6881\mathrm{e}{-14}\)
		& \(2\) & \(0.0368\) & \(2.3636\mathrm{e}{-16}\) & \(1.6980\mathrm{e}{-14}\) \\
		\(150\)
		& \(2\) & \(0.0974\) & \(1.8141\mathrm{e}{-15}\) & \(2.2149\mathrm{e}{-14}\)
		& \(2\) & \(0.0917\) & \(3.3485\mathrm{e}{-16}\) & \(2.5504\mathrm{e}{-14}\)
		& \(2\) & \(0.0914\) & \(2.4456\mathrm{e}{-16}\) & \(2.5775\mathrm{e}{-14}\) \\
		\(200\)
		& \(2\) & \(0.1708\) & \(2.1225\mathrm{e}{-15}\) & \(3.0834\mathrm{e}{-14}\)
		& \(2\) & \(0.2348\) & \(3.3605\mathrm{e}{-16}\) & \(3.4051\mathrm{e}{-14}\)
		& \(2\) & \(0.2185\) & \(2.4044\mathrm{e}{-16}\) & \(3.4193\mathrm{e}{-14}\) \\
		\(250\)
		& \(2\) & \(0.3243\) & \(2.3221\mathrm{e}{-15}\) & \(4.8257\mathrm{e}{-14}\)
		& \(2\) & \(0.2899\) & \(3.3890\mathrm{e}{-16}\) & \(4.2272\mathrm{e}{-14}\)
		& \(2\) & \(0.317\) & \(2.4719\mathrm{e}{-16}\) & \(4.3391\mathrm{e}{-14}\) \\
		
		\midrule
		
		\(\mathrm{GGS}\)
		& \multicolumn{4}{c}{\(\mu=0\)}
		& \multicolumn{4}{c}{\(\mu=2\)}
		& \multicolumn{4}{c}{\(\mu=4\)} \\
		\cmidrule(lr){2-5}
		\cmidrule(lr){6-9}
		\cmidrule(lr){10-13}
		\(m\)
		& \(\mathrm{IT}\) & \(\mathrm{CPU}\) & \(\mathrm{RES}\) & \(\mathrm{ERR}\)
		& \(\mathrm{IT}\) & \(\mathrm{CPU}\) & \(\mathrm{RES}\) & \(\mathrm{ERR}\)
		& \(\mathrm{IT}\) & \(\mathrm{CPU}\) & \(\mathrm{RES}\) & \(\mathrm{ERR}\) \\
		\midrule
		\(10\)
		& \(55\) & \(0.0297\) & \(8.0964\mathrm{e}{-07}\) & \(3.6496\mathrm{e}{-06}\)
		& \(21\) & \(0.0124\) & \(8.7252\mathrm{e}{-07}\) & \(3.9212\mathrm{e}{-06}\)
		& \(15\) & \(0.0086\) & \(6.0533\mathrm{e}{-07}\) & \(3.3143\mathrm{e}{-06}\) \\
		\(50\)
		& \(201\) & \(41.024\) & \(9.1972\mathrm{e}{-07}\) & \(1.4129\mathrm{e}{-05}\)
		& \(27\) & \(1.031\) & \(7.3293\mathrm{e}{-07}\) & \(2.2636\mathrm{e}{-05}\)
		& \(17\) & \(0.5858\) & \(4.4997\mathrm{e}{-07}\) & \(1.3659\mathrm{e}{-05}\) \\
		\(100\)
		& \(368\) & \(1031.3\) & \(9.3012\mathrm{e}{-07}\) & \(2.5205\mathrm{e}{-05}\)
		& \(27\) & \(13.808\) & \(8.7454\mathrm{e}{-07}\) & \(5.3352\mathrm{e}{-05}\)
		& \(17\) & \(8.5943\) & \(5.0198\mathrm{e}{-07}\) & \(3.0316\mathrm{e}{-05}\) \\
		\(150\)
		& \(\mathrm{Fail}\) & \(--\) & \(--\) & \(--\)
		& \(27\) & \(80.803\) & \(9.2332\mathrm{e}{-07}\) & \(8.4059\mathrm{e}{-05}\)
		& \(17\) & \(51.21\) & \(5.1957\mathrm{e}{-07}\) & \(4.6971\mathrm{e}{-05}\) \\
		\(200\)
		& \(\mathrm{Fail}\) & \(--\) & \(--\) & \(--\)
		& \(27\) & \(278.62\) & \(9.4804\mathrm{e}{-07}\) & \(1.1477\mathrm{e}{-04}\)
		& \(17\) & \(179.25\) & \(5.2841\mathrm{e}{-07}\) & \(6.3625\mathrm{e}{-05}\) \\
		\(250\)
		& \(\mathrm{Fail}\) & \(--\) & \(--\) & \(--\)
		& \(27\) & \(640.7\) & \(9.6297\mathrm{e}{-07}\) & \(1.4547\mathrm{e}{-04}\)
		& \(17\) & \(406.33\) & \(5.3373\mathrm{e}{-07}\) & \(8.0280\mathrm{e}{-05}\) \\
		
		\midrule
		
		\(\mathrm{FPI\mbox{-}SS}\)
		& \multicolumn{4}{c}{
			\makecell[c]{\(\mu=0\)\\
				\((\alpha=4,\omega=0.55)\)}}
		& \multicolumn{4}{c}{
			\makecell[c]{\(\mu=2\)\\
				\((\alpha=11.6,\omega=0.85)\)}}
		& \multicolumn{4}{c}{
			\makecell[c]{\(\mu=4\)\\
				\((\alpha=17.9,\omega=0.95)\)}} \\
		\cmidrule(lr){2-5}
		\cmidrule(lr){6-9}
		\cmidrule(lr){10-13}
		\(m\)
		& \(\mathrm{IT}\) & \(\mathrm{CPU}\) & \(\mathrm{RES}\) & \(\mathrm{ERR}\)
		& \(\mathrm{IT}\) & \(\mathrm{CPU}\) & \(\mathrm{RES}\) & \(\mathrm{ERR}\)
		& \(\mathrm{IT}\) & \(\mathrm{CPU}\) & \(\mathrm{RES}\) & \(\mathrm{ERR}\) \\
		\midrule
		\(10\)
		& \(36\) & \(0.0006\) & \(9.5939\mathrm{e}{-07}\) & \(6.2695\mathrm{e}{-06}\)
		& \(16\) & \(0.0003\) & \(6.2858\mathrm{e}{-07}\) & \(2.6560\mathrm{e}{-06}\)
		& \(12\) & \(0.0002\) & \(7.7670\mathrm{e}{-07}\) & \(2.2928\mathrm{e}{-06}\) \\
		\(50\)
		& \(125\) & \(0.1558\) & \(7.7662\mathrm{e}{-07}\) & \(1.7607\mathrm{e}{-05}\)
		& \(18\) & \(0.0046\) & \(4.3208\mathrm{e}{-07}\) & \(1.3281\mathrm{e}{-05}\)
		& \(13\) & \(0.003\) & \(2.6535\mathrm{e}{-07}\) & \(7.8371\mathrm{e}{-06}\) \\
		\(100\)
		& \(221\) & \(1.0459\) & \(9.9758\mathrm{e}{-07}\) & \(4.0063\mathrm{e}{-05}\)
		& \(18\) & \(0.0249\) & \(5.3319\mathrm{e}{-07}\) & \(3.2498\mathrm{e}{-05}\)
		& \(13\) & \(0.0175\) & \(2.9726\mathrm{e}{-07}\) & \(1.7864\mathrm{e}{-05}\) \\
		\(150\)
		& \(315\) & \(1.1245\) & \(8.9020\mathrm{e}{-07}\) & \(5.0474\mathrm{e}{-05}\)
		& \(18\) & \(0.0657\) & \(5.6879\mathrm{e}{-07}\) & \(5.1763\mathrm{e}{-05}\)
		& \(13\) & \(0.0485\) & \(3.0926\mathrm{e}{-07}\) & \(2.7908\mathrm{e}{-05}\) \\
		\(200\)
		& \(406\) & \(3.0748\) & \(9.8758\mathrm{e}{-07}\) & \(7.2456\mathrm{e}{-05}\)
		& \(18\) & \(0.15\) & \(5.8693\mathrm{e}{-07}\) & \(7.1037\mathrm{e}{-05}\)
		& \(13\) & \(0.0956\) & \(3.1549\mathrm{e}{-07}\) & \(3.7955\mathrm{e}{-05}\) \\
		\(250\)
		& \(497\) & \(6.5506\) & \(9.1634\mathrm{e}{-07}\) & \(8.1943\mathrm{e}{-05}\)
		& \(18\) & \(0.2333\) & \(5.9793\mathrm{e}{-07}\) & \(9.0314\mathrm{e}{-05}\)
		& \(13\) & \(0.1618\) & \(3.1930\mathrm{e}{-07}\) & \(4.8004\mathrm{e}{-05}\) \\
		
		\midrule
		
		\(\mathbf{RMNSI}\)
		& \multicolumn{4}{c}{
			\makecell[c]{\(\mu=0\)\\ \((\theta=1)\)}}
		& \multicolumn{4}{c}{
			\makecell[c]{\(\mu=2\)\\ \((\theta=1.3)\)}}
		& \multicolumn{4}{c}{
			\makecell[c]{\(\mu=4\)\\ \((\theta=1.4)\)}} \\
		\cmidrule(lr){2-5}
		\cmidrule(lr){6-9}
		\cmidrule(lr){10-13}
		\(m\)
		& \(\mathrm{IT}\) & \(\mathrm{CPU}\) & \(\mathrm{RES}\) & \(\mathrm{ERR}\)
		& \(\mathrm{IT}\) & \(\mathrm{CPU}\) & \(\mathrm{RES}\) & \(\mathrm{ERR}\)
		& \(\mathrm{IT}\) & \(\mathrm{CPU}\) & \(\mathrm{RES}\) & \(\mathrm{ERR}\) \\
		\midrule
		\(10\)
		& \(1\) & \(<0.0001\) & \(7.6833\mathrm{e}{-16}\) & \(2.0502\mathrm{e}{-15}\)
		& \(1\) & \(<0.0001\) & \(3.3321\mathrm{e}{-16}\) & \(1.1213\mathrm{e}{-15}\)
		& \(1\) & \(<0.0001\) & \(2.0177\mathrm{e}{-16}\) & \(1.3461\mathrm{e}{-15}\) \\
		\(50\)
		& \(1\) & \(0.0016\) & \(1.0553\mathrm{e}{-15}\) & \(6.5390\mathrm{e}{-15}\)
		& \(1\) & \(0.0008\) & \(3.1954\mathrm{e}{-16}\) & \(8.2769\mathrm{e}{-15}\)
		& \(1\) & \(0.0007\) & \(2.3034\mathrm{e}{-16}\) & \(8.4018\mathrm{e}{-15}\) \\
		\(100\)
		& \(1\) & \(0.0068\) & \(1.4490\mathrm{e}{-15}\) & \(2.1534\mathrm{e}{-14}\)
		& \(1\) & \(0.0032\) & \(3.3583\mathrm{e}{-16}\) & \(1.6881\mathrm{e}{-14}\)
		& \(1\) & \(0.0032\) & \(2.3636\mathrm{e}{-16}\) & \(1.6980\mathrm{e}{-14}\) \\
		\(150\)
		& \(1\) & \(0.007\) & \(1.8141\mathrm{e}{-15}\) & \(2.2149\mathrm{e}{-14}\)
		& \(1\) & \(0.0072\) & \(3.3485\mathrm{e}{-16}\) & \(2.5504\mathrm{e}{-14}\)
		& \(1\) & \(0.0061\) & \(2.4456\mathrm{e}{-16}\) & \(2.5775\mathrm{e}{-14}\) \\
		\(200\)
		& \(1\) & \(0.0179\) & \(2.1225\mathrm{e}{-15}\) & \(3.0834\mathrm{e}{-14}\)
		& \(1\) & \(0.0161\) & \(3.3605\mathrm{e}{-16}\) & \(3.4051\mathrm{e}{-14}\)
		& \(1\) & \(0.0157\) & \(2.4044\mathrm{e}{-16}\) & \(3.4193\mathrm{e}{-14}\) \\
		\(250\)
		& \(1\) & \(0.0233\) & \(2.3221\mathrm{e}{-15}\) & \(4.8257\mathrm{e}{-14}\)
		& \(1\) & \(0.0267\) & \(3.3890\mathrm{e}{-16}\) & \(4.2272\mathrm{e}{-14}\)
		& \(1\) & \(0.0265\) & \(2.4719\mathrm{e}{-16}\) & \(4.3391\mathrm{e}{-14}\) \\
		
		\bottomrule
	\end{tabular*}
\end{table}

The results for Example~\ref{Ex2} are reported in
Table~\ref{tab:ex2-comparison}. For $\mu=0$, the underlying matrix $V$
is less diagonally dominant. Consequently, GGS fails for larger problem
sizes, while FPI-SS requires increasing computational effort as $m$
grows. Positive values of $\mu$ strengthen the diagonal dominance and
improve the performance of both methods. In contrast, RMNSI performs
consistently for all tested values of $m$ and $\mu$ and produces highly
accurate solutions. It also gives better overall efficiency than FPM-M
and M3.

For Example~\ref{Ex2}, Figure~\ref{fig:theta-ex2} presents a stepwise
decrease in the iteration count as $\theta$ increases. Although the
transition points vary with $\mu$, their independence from $m$ is
confirmed by the complete overlap of the curves.

\begin{example}\label{Ex3} (\cite{WangMu2025}). Let $H_n=(h_{ij})\in\mathbb{R}^{n\times n}$ be
	the Hilbert matrix defined by
	$h_{ij}=1/(i+j-1)$, $1\leq i,j\leq n$. In
	GAVE~\eqref{eq:GAVE}, take
	\begin{equation*}
		A=H_n+2I_n,\qquad B=I_n.
	\end{equation*}
	The prescribed exact solution is
	$x_i^*=(-1)^i$, $i=1,2,\ldots,n$, that is,
$
		x^*=(-1,1,-1,1,\ldots)^{\mathsf T}.
	$
	Since $|x^*|=\boldsymbol{e}$, the right-hand side is constructed as
	$
		b=Ax^*-B|x^*|
		=(H_n+2I_n)x^*-\boldsymbol{e}.
$
\end{example}

\begin{table}[width=\linewidth,cols=8,pos=H]
	\caption{\(\mathrm{Performance~comparison~of~tested~methods~for~Example~\ref{Ex3}.}\)}
	\label{tab:ex3-comparison}
	\footnotesize
	\setlength{\tabcolsep}{2pt}
	\renewcommand{\arraystretch}{0.85}
	\begin{tabular*}{\tblwidth}
		{@{\extracolsep{\fill}}llcccccc@{}}
		\toprule
		\(\mathrm{Method}\) &
		& \(n=5000\) & \(n=10000\) & \(n=20000\)
		& \(n=30000\) & \(n=40000\) & \(n=50000\) \\
		\midrule
		
		\multirow{4}{*}{\(\mathrm{FPM\mbox{-}M}\)}
		& \(\mathrm{IT}\) & \(21\) & \(21\) & \(21\) & \(21\) & \(21\) & \(21\) \\
		& \(\mathrm{CPU}\) & \(0.097\) & \(0.172\) & \(0.4525\) & \(0.9011\) & \(0.9171\) & \(1.4379\) \\
		& \(\mathrm{RES}\) & \(7.2904\mathrm{e}{-07}\) & \(7.2781\mathrm{e}{-07}\) & \(7.2721\mathrm{e}{-07}\) & \(7.2705\mathrm{e}{-07}\) & \(7.2699\mathrm{e}{-07}\) & \(7.2697\mathrm{e}{-07}\) \\
		& \(\mathrm{ERR}\) & \(9.1803\mathrm{e}{-05}\) & \(1.2954\mathrm{e}{-04}\) & \(1.8305\mathrm{e}{-04}\) & \(2.2415\mathrm{e}{-04}\) & \(2.5882\mathrm{e}{-04}\) & \(2.8937\mathrm{e}{-04}\) \\
		\addlinespace
		
		\multirow{4}{*}{\(\mathrm{M3}\)}
		& \(\mathrm{IT}\) & \(50\) & \(50\) & \(50\) & \(50\) & \(50\) & \(50\) \\
		& \(\mathrm{CPU}\) & \(0.2384\) & \(0.3784\) & \(0.8362\) & \(1.7407\) & \(1.7497\) & \(3.1645\) \\
		& \(\mathrm{RES}\) & \(9.2240\mathrm{e}{-07}\) & \(9.2254\mathrm{e}{-07}\) & \(9.2264\mathrm{e}{-07}\) & \(9.2266\mathrm{e}{-07}\) & \(9.2267\mathrm{e}{-07}\) & \(9.2267\mathrm{e}{-07}\) \\
		& \(\mathrm{ERR}\) & \(1.3791\mathrm{e}{-04}\) & \(1.9507\mathrm{e}{-04}\) & \(2.7588\mathrm{e}{-04}\) & \(3.3788\mathrm{e}{-04}\) & \(3.9015\mathrm{e}{-04}\) & \(4.3620\mathrm{e}{-04}\) \\
		\addlinespace
		
		\multirow{4}{*}{\(\mathrm{GGS}\)}
		& \(\mathrm{IT}\) & \(13\) & \(13\) & \(13\) & \(13\) & \(13\) & \(14\) \\
		& \(\mathrm{CPU}\) & \(0.9537\) & \(3.5852\) & \(14.046\) & \(31.931\) & \(56.737\) & \(96.186\) \\
		& \(\mathrm{RES}\) & \(4.8728\mathrm{e}{-07}\) & \(6.7027\mathrm{e}{-07}\) & \(8.1927\mathrm{e}{-07}\) & \(8.9590\mathrm{e}{-07}\) & \(9.5593\mathrm{e}{-07}\) & \(4.1080\mathrm{e}{-07}\) \\
		& \(\mathrm{ERR}\) & \(2.9283\mathrm{e}{-05}\) & \(5.9370\mathrm{e}{-05}\) & \(1.1011\mathrm{e}{-04}\) & \(1.5387\mathrm{e}{-04}\) & \(1.9397\mathrm{e}{-04}\) & \(8.0281\mathrm{e}{-05}\) \\
		\addlinespace
		
		\multirow{4}{*}{\makecell[l]{
				\(\mathrm{FPI\mbox{-}SS}\)\\
				\((\alpha=2.9,\omega=1)\)}}
		& \(\mathrm{IT}\) & \(17\) & \(17\) & \(17\) & \(17\) & \(17\) & \(17\) \\
		& \(\mathrm{CPU}\) & \(0.0882\) & \(0.1581\) & \(0.4496\) & \(0.8674\) & \(0.8973\) & \(1.4431\) \\
		& \(\mathrm{RES}\) & \(7.9965\mathrm{e}{-07}\) & \(7.9889\mathrm{e}{-07}\) & \(7.9918\mathrm{e}{-07}\) & \(7.9932\mathrm{e}{-07}\) & \(7.9943\mathrm{e}{-07}\) & \(7.9949\mathrm{e}{-07}\) \\
		& \(\mathrm{ERR}\) & \(1.0603\mathrm{e}{-04}\) & \(1.4982\mathrm{e}{-04}\) & \(2.1198\mathrm{e}{-04}\) & \(2.5967\mathrm{e}{-04}\) & \(2.9987\mathrm{e}{-04}\) & \(3.3529\mathrm{e}{-04}\) \\
		\addlinespace
		
	\multirow{4}{*}{\makecell[l]{
			\(\mathbf{RMNSI}\)\\
			\((\theta=2.3)\)}}
	& \(\mathrm{IT}\) & \(5\) & \(5\) & \(6\) & \(6\) & \(6\) & \(6\) \\
	& \(\mathrm{CPU}\) & \(0.0343\) & \(0.0644\) & \(0.2264\) & \(0.4207\) & \(0.4406\) & \(0.7131\) \\
	& \(\mathrm{RES}\) & \(8.9865\mathrm{e}{-07}\) & \(9.5339\mathrm{e}{-07}\) & \(9.3841\mathrm{e}{-08}\) & \(9.5224\mathrm{e}{-08}\) & \(9.6106\mathrm{e}{-08}\) & \(9.6753\mathrm{e}{-08}\) \\
	& \(\mathrm{ERR}\) & \(1.1315\mathrm{e}{-04}\) & \(1.6564\mathrm{e}{-04}\) & \(2.5470\mathrm{e}{-05}\) & \(3.1360\mathrm{e}{-05}\) & \(3.6318\mathrm{e}{-05}\) & \(4.0690\mathrm{e}{-05}\) \\
		
		\bottomrule
	\end{tabular*}
\end{table}

Table~\ref{tab:ex3-comparison} presents the results for the dense
mixed-sign problem in Example~\ref{Ex3}. All the tested methods reach the
prescribed tolerance, but RMNSI uses the fewest updates and the shortest
running time throughout. GGS produces smaller errors, although its cost
increases sharply with the problem dimension. FPM-M and FPI-SS show
similar efficiency, whereas M3 is comparatively slower. Thus, RMNSI
offers the most favourable balance between speed and accuracy for this
example.

\begin{figure}[width=\linewidth,pos=H]
	\centering
	\includegraphics[width=0.5\linewidth]{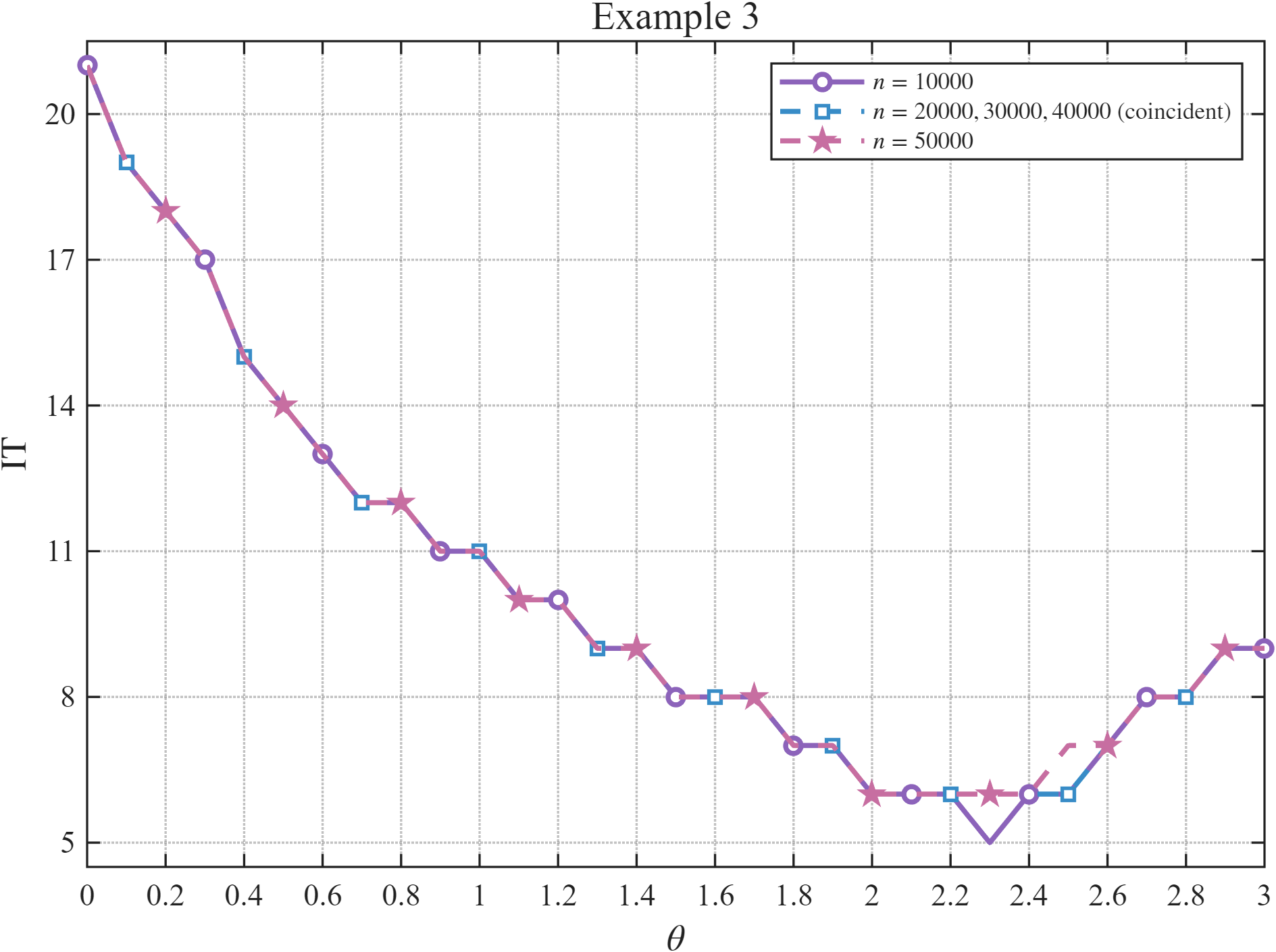}
	\caption{\(\mathrm{Effect~of~the~relaxation~parameter~\theta~on~the~iteration~
			count~of~RMNSI~for~Example~\ref{Ex3}~with~different~values~of~n.}\)}
	\label{fig:theta-ex3}
\end{figure}

Figure~\ref{fig:theta-ex3} traces the parameter sensitivity of RMNSI for
the dense mixed-sign problem in Example~\ref{Ex3}. The nearly identical
curves decrease to a common low-iteration region and then rise, showing
that intermediate values of $\theta$ are preferable for every tested
dimension.

\begin{example}\label{Ex4} (\cite{LuoEtAl2026}). Let
	$
		V=\mathcal{V}_{m,\mu}(-1,4,-1;-1,-1),
	~
		D=0.8\operatorname{diag}(V),
$
	and define
	$
		z_i^*=
		\begin{cases}
			1, & i\text{ is odd},\\
			2, & i\text{ is even}.
		\end{cases}
$
	Setting $q=-Vz^*$ gives $w^*=0$, and therefore $(z^*,w^*)$
	satisfies the complementarity conditions.
	
	Using the scaled formulation \eqref{eq:LCP-GAVE-2},
	GAVE~\eqref{eq:GAVE} is defined by
	$A=V+D$, $B=D-V$, and $b=-q$. The associated exact solution is
$
		x^*=\frac{1}{2}z^*
		=
		\left(
		\frac{1}{2},1,\frac{1}{2},1,\ldots
		\right)^{\mathsf T}.
$
\end{example}

\begin{table}[width=\linewidth,cols=13,pos=H]
	\caption{\(\mathrm{Performance~comparison~of~tested~methods~
			for~Example~\ref{Ex4}.}\)}
	\label{tab:ex4-comparison}
	\footnotesize
	\setlength{\tabcolsep}{1.2pt}
	\renewcommand{\arraystretch}{0.9}
	
	\begin{tabular*}{\tblwidth}
		{@{\extracolsep{\fill}}l*{12}{c}@{}}
		\toprule
		
		\(\mathrm{FPM\mbox{-}M}\)
		& \multicolumn{4}{c}{\(\mu=0\)}
		& \multicolumn{4}{c}{\(\mu=2\)}
		& \multicolumn{4}{c}{\(\mu=4\)} \\
		\cmidrule(lr){2-5}
		\cmidrule(lr){6-9}
		\cmidrule(lr){10-13}
		\(m\)
		& \(\mathrm{IT}\) & \(\mathrm{CPU}\) & \(\mathrm{RES}\) & \(\mathrm{ERR}\)
		& \(\mathrm{IT}\) & \(\mathrm{CPU}\) & \(\mathrm{RES}\) & \(\mathrm{ERR}\)
		& \(\mathrm{IT}\) & \(\mathrm{CPU}\) & \(\mathrm{RES}\) & \(\mathrm{ERR}\) \\
		\midrule
		\(10\)
		& \(\mathrm{Fail}\) & \(--\) & \(--\) & \(--\)
		& \(\mathrm{Fail}\) & \(--\) & \(--\) & \(--\)
		& \(50\) & \(0.0006\) & \(8.7147\mathrm{e}{-07}\) & \(2.9180\mathrm{e}{-06}\) \\
		\(50\)
		& \(\mathrm{Fail}\) & \(--\) & \(--\) & \(--\)
		& \(\mathrm{Fail}\) & \(--\) & \(--\) & \(--\)
		& \(58\) & \(0.0039\) & \(8.7355\mathrm{e}{-07}\) & \(1.3417\mathrm{e}{-05}\) \\
		\(100\)
		& \(\mathrm{Fail}\) & \(--\) & \(--\) & \(--\)
		& \(\mathrm{Fail}\) & \(--\) & \(--\) & \(--\)
		& \(56\) & \(0.0233\) & \(8.8271\mathrm{e}{-07}\) & \(2.6904\mathrm{e}{-05}\) \\
		\(150\)
		& \(\mathrm{Fail}\) & \(--\) & \(--\) & \(--\)
		& \(\mathrm{Fail}\) & \(--\) & \(--\) & \(--\)
		& \(54\) & \(0.0749\) & \(9.8685\mathrm{e}{-07}\) & \(4.5007\mathrm{e}{-05}\) \\
		\(200\)
		& \(\mathrm{Fail}\) & \(--\) & \(--\) & \(--\)
		& \(\mathrm{Fail}\) & \(--\) & \(--\) & \(--\)
		& \(53\) & \(0.1304\) & \(9.9999\mathrm{e}{-07}\) & \(6.0734\mathrm{e}{-05}\) \\
		\(250\)
		& \(\mathrm{Fail}\) & \(--\) & \(--\) & \(--\)
		& \(\mathrm{Fail}\) & \(--\) & \(--\) & \(--\)
		& \(53\) & \(0.1923\) & \(8.9931\mathrm{e}{-07}\) & \(6.8219\mathrm{e}{-05}\) \\
		
		\midrule
		
		\(\mathrm{M3}\)
		& \multicolumn{4}{c}{\(\mu=0\)}
		& \multicolumn{4}{c}{\(\mu=2\)}
		& \multicolumn{4}{c}{\(\mu=4\)} \\
		\cmidrule(lr){2-5}
		\cmidrule(lr){6-9}
		\cmidrule(lr){10-13}
		\(m\)
		& \(\mathrm{IT}\) & \(\mathrm{CPU}\) & \(\mathrm{RES}\) & \(\mathrm{ERR}\)
		& \(\mathrm{IT}\) & \(\mathrm{CPU}\) & \(\mathrm{RES}\) & \(\mathrm{ERR}\)
		& \(\mathrm{IT}\) & \(\mathrm{CPU}\) & \(\mathrm{RES}\) & \(\mathrm{ERR}\) \\
		\midrule
		\(10\)
		& \(17\) & \(0.0006\) & \(5.3370\mathrm{e}{-16}\) & \(1.1748\mathrm{e}{-14}\)
		& \(3\) & \(0.0007\) & \(4.8514\mathrm{e}{-16}\) & \(2.4248\mathrm{e}{-15}\)
		& \(2\) & \(0.0002\) & \(2.3850\mathrm{e}{-16}\) & \(1.1551\mathrm{e}{-15}\) \\
		\(50\)
		& \(67\) & \(0.0123\) & \(7.0442\mathrm{e}{-16}\) & \(6.7175\mathrm{e}{-13}\)
		& \(3\) & \(0.0026\) & \(6.4485\mathrm{e}{-16}\) & \(1.9171\mathrm{e}{-14}\)
		& \(2\) & \(0.0025\) & \(3.4687\mathrm{e}{-16}\) & \(8.0789\mathrm{e}{-15}\) \\
		\(100\)
		& \(138\) & \(0.1109\) & \(6.0990\mathrm{e}{-16}\) & \(1.1729\mathrm{e}{-11}\)
		& \(3\) & \(0.0173\) & \(5.0055\mathrm{e}{-16}\) & \(2.5016\mathrm{e}{-14}\)
		& \(2\) & \(0.0164\) & \(3.1635\mathrm{e}{-16}\) & \(1.4450\mathrm{e}{-14}\) \\
		\(150\)
		& \(296\) & \(0.4774\) & \(6.2048\mathrm{e}{-16}\) & \(3.8675\mathrm{e}{-11}\)
		& \(3\) & \(0.0394\) & \(5.2303\mathrm{e}{-16}\) & \(3.8670\mathrm{e}{-14}\)
		& \(2\) & \(0.0416\) & \(3.2099\mathrm{e}{-16}\) & \(2.1834\mathrm{e}{-14}\) \\
		\(200\)
		& \(\mathrm{Fail}\) & \(--\) & \(--\) & \(--\)
		& \(3\) & \(0.073\) & \(5.1646\mathrm{e}{-16}\) & \(4.9273\mathrm{e}{-14}\)
		& \(2\) & \(0.0699\) & \(3.2976\mathrm{e}{-16}\) & \(3.0020\mathrm{e}{-14}\) \\
		\(250\)
		& \(\mathrm{Fail}\) & \(--\) & \(--\) & \(--\)
		& \(3\) & \(0.1331\) & \(5.2764\mathrm{e}{-16}\) & \(6.2389\mathrm{e}{-14}\)
		& \(2\) & \(0.1165\) & \(3.2770\mathrm{e}{-16}\) & \(3.7212\mathrm{e}{-14}\) \\
		
		\midrule
		
		\(\mathrm{GGS}\)
		& \multicolumn{4}{c}{\(\mu=0\)}
		& \multicolumn{4}{c}{\(\mu=2\)}
		& \multicolumn{4}{c}{\(\mu=4\)} \\
		\cmidrule(lr){2-5}
		\cmidrule(lr){6-9}
		\cmidrule(lr){10-13}
		\(m\)
		& \(\mathrm{IT}\) & \(\mathrm{CPU}\) & \(\mathrm{RES}\) & \(\mathrm{ERR}\)
		& \(\mathrm{IT}\) & \(\mathrm{CPU}\) & \(\mathrm{RES}\) & \(\mathrm{ERR}\)
		& \(\mathrm{IT}\) & \(\mathrm{CPU}\) & \(\mathrm{RES}\) & \(\mathrm{ERR}\) \\
		\midrule
		\(10\)
		& \(126\) & \(0.1818\) & \(9.9963\mathrm{e}{-07}\) & \(6.7583\mathrm{e}{-05}\)
		& \(16\) & \(0.0419\) & \(8.9024\mathrm{e}{-07}\) & \(9.0553\mathrm{e}{-06}\)
		& \(11\) & \(0.0072\) & \(6.1225\mathrm{e}{-07}\) & \(5.4570\mathrm{e}{-06}\) \\
		\(50\)
		& \(\mathrm{Fail}\) & \(--\) & \(--\) & \(--\)
		& \(18\) & \(2.6456\) & \(7.3903\mathrm{e}{-07}\) & \(3.9734\mathrm{e}{-05}\)
		& \(12\) & \(0.4118\) & \(4.5080\mathrm{e}{-07}\) & \(2.0477\mathrm{e}{-05}\) \\
		\(100\)
		& \(\mathrm{Fail}\) & \(--\) & \(--\) & \(--\)
		& \(18\) & \(37.883\) & \(8.1810\mathrm{e}{-07}\) & \(8.7423\mathrm{e}{-05}\)
		& \(12\) & \(6.1773\) & \(4.8608\mathrm{e}{-07}\) & \(4.4000\mathrm{e}{-05}\) \\
		\(150\)
		& \(\mathrm{Fail}\) & \(--\) & \(--\) & \(--\)
		& \(18\) & \(194.51\) & \(8.4488\mathrm{e}{-07}\) & \(1.3511\mathrm{e}{-04}\)
		& \(12\) & \(36.177\) & \(4.9796\mathrm{e}{-07}\) & \(6.7522\mathrm{e}{-05}\) \\
		\(200\)
		& \(\mathrm{Fail}\) & \(--\) & \(--\) & \(--\)
		& \(18\) & \(548.38\) & \(8.5836\mathrm{e}{-07}\) & \(1.8279\mathrm{e}{-04}\)
		& \(12\) & \(124.91\) & \(5.0392\mathrm{e}{-07}\) & \(9.1043\mathrm{e}{-05}\) \\
		\(250\)
		& \(\mathrm{Fail}\) & \(--\) & \(--\) & \(--\)
		& \(18\) & \(617.2\) & \(8.6647\mathrm{e}{-07}\) & \(2.3048\mathrm{e}{-04}\)
		& \(12\) & \(283.34\) & \(5.0750\mathrm{e}{-07}\) & \(1.1456\mathrm{e}{-04}\) \\
		
		\midrule
		
		\(\mathrm{FPI\mbox{-}SS}\)
		& \multicolumn{4}{c}{
			\makecell[c]{\(\mu=0\)\\
				\((\alpha=4.3,\omega=0.9)\)}}
		& \multicolumn{4}{c}{
			\makecell[c]{\(\mu=2\)\\
				\((\alpha=9.7,\omega=0.95)\)}}
		& \multicolumn{4}{c}{
			\makecell[c]{\(\mu=4\)\\
				\((\alpha=13.4,\omega=0.9)\)}} \\
		\cmidrule(lr){2-5}
		\cmidrule(lr){6-9}
		\cmidrule(lr){10-13}
		\(m\)
		& \(\mathrm{IT}\) & \(\mathrm{CPU}\) & \(\mathrm{RES}\) & \(\mathrm{ERR}\)
		& \(\mathrm{IT}\) & \(\mathrm{CPU}\) & \(\mathrm{RES}\) & \(\mathrm{ERR}\)
		& \(\mathrm{IT}\) & \(\mathrm{CPU}\) & \(\mathrm{RES}\) & \(\mathrm{ERR}\) \\
		\midrule
		\(10\)
		& \(129\) & \(0.0078\) & \(9.9641\mathrm{e}{-07}\) & \(6.9247\mathrm{e}{-05}\)
		& \(18\) & \(0.0008\) & \(9.6034\mathrm{e}{-07}\) & \(7.8722\mathrm{e}{-06}\)
		& \(12\) & \(0.0001\) & \(8.2668\mathrm{e}{-07}\) & \(7.0993\mathrm{e}{-06}\) \\
		\(50\)
		& \(\mathrm{Fail}\) & \(--\) & \(--\) & \(--\)
		& \(20\) & \(0.0178\) & \(6.2898\mathrm{e}{-07}\) & \(3.3158\mathrm{e}{-05}\)
		& \(13\) & \(0.0032\) & \(5.8764\mathrm{e}{-07}\) & \(2.6570\mathrm{e}{-05}\) \\
		\(100\)
		& \(\mathrm{Fail}\) & \(--\) & \(--\) & \(--\)
		& \(20\) & \(0.0913\) & \(6.8489\mathrm{e}{-07}\) & \(7.2543\mathrm{e}{-05}\)
		& \(13\) & \(0.0156\) & \(6.2746\mathrm{e}{-07}\) & \(5.6675\mathrm{e}{-05}\) \\
		\(150\)
		& \(\mathrm{Fail}\) & \(--\) & \(--\) & \(--\)
		& \(20\) & \(0.1127\) & \(7.0394\mathrm{e}{-07}\) & \(1.1193\mathrm{e}{-04}\)
		& \(13\) & \(0.0419\) & \(6.4088\mathrm{e}{-07}\) & \(8.6781\mathrm{e}{-05}\) \\
		\(200\)
		& \(\mathrm{Fail}\) & \(--\) & \(--\) & \(--\)
		& \(20\) & \(0.3926\) & \(7.1354\mathrm{e}{-07}\) & \(1.5131\mathrm{e}{-04}\)
		& \(13\) & \(0.0924\) & \(6.4762\mathrm{e}{-07}\) & \(1.1689\mathrm{e}{-04}\) \\
		\(250\)
		& \(\mathrm{Fail}\) & \(--\) & \(--\) & \(--\)
		& \(20\) & \(0.2569\) & \(7.1932\mathrm{e}{-07}\) & \(1.9070\mathrm{e}{-04}\)
		& \(13\) & \(0.156\) & \(6.5168\mathrm{e}{-07}\) & \(1.4699\mathrm{e}{-04}\) \\
		
		\midrule
		
		\(\mathbf{RMNSI}\)
		& \multicolumn{4}{c}{
			\makecell[c]{\(\mu=0\)\\ \((\theta=0.7)\)}}
		& \multicolumn{4}{c}{
			\makecell[c]{\(\mu=2\)\\ \((\theta=0.7)\)}}
		& \multicolumn{4}{c}{
			\makecell[c]{\(\mu=4\)\\ \((\theta=0.7)\)}} \\
		\cmidrule(lr){2-5}
		\cmidrule(lr){6-9}
		\cmidrule(lr){10-13}
		\(m\)
		& \(\mathrm{IT}\) & \(\mathrm{CPU}\) & \(\mathrm{RES}\) & \(\mathrm{ERR}\)
		& \(\mathrm{IT}\) & \(\mathrm{CPU}\) & \(\mathrm{RES}\) & \(\mathrm{ERR}\)
		& \(\mathrm{IT}\) & \(\mathrm{CPU}\) & \(\mathrm{RES}\) & \(\mathrm{ERR}\) \\
		\midrule
		\(10\)
		& \(123\) & \(0.0074\) & \(9.4772\mathrm{e}{-07}\) & \(4.3548\mathrm{e}{-05}\)
		& \(13\) & \(0.0008\) & \(6.5556\mathrm{e}{-07}\) & \(6.6995\mathrm{e}{-06}\)
		& \(8\) & \(0.0001\) & \(4.3343\mathrm{e}{-07}\) & \(3.8448\mathrm{e}{-06}\) \\
		\(50\)
		& \(\mathrm{Fail}\) & \(--\) & \(--\) & \(--\)
		& \(14\) & \(0.0069\) & \(8.4777\mathrm{e}{-07}\) & \(4.4434\mathrm{e}{-05}\)
		& \(8\) & \(0.0009\) & \(9.5122\mathrm{e}{-07}\) & \(4.3025\mathrm{e}{-05}\) \\
		\(100\)
		& \(\mathrm{Fail}\) & \(--\) & \(--\) & \(--\)
		& \(14\) & \(0.0296\) & \(9.2137\mathrm{e}{-07}\) & \(9.7290\mathrm{e}{-05}\)
		& \(9\) & \(0.0074\) & \(2.1426\mathrm{e}{-07}\) & \(1.9327\mathrm{e}{-05}\) \\
		\(150\)
		& \(\mathrm{Fail}\) & \(--\) & \(--\) & \(--\)
		& \(14\) & \(0.0415\) & \(9.4626\mathrm{e}{-07}\) & \(1.5014\mathrm{e}{-04}\)
		& \(9\) & \(0.0192\) & \(2.1925\mathrm{e}{-07}\) & \(2.9662\mathrm{e}{-05}\) \\
		\(200\)
		& \(\mathrm{Fail}\) & \(--\) & \(--\) & \(--\)
		& \(14\) & \(0.1484\) & \(9.5878\mathrm{e}{-07}\) & \(2.0300\mathrm{e}{-04}\)
		& \(9\) & \(0.0381\) & \(2.2175\mathrm{e}{-07}\) & \(3.9996\mathrm{e}{-05}\) \\
		\(250\)
		& \(\mathrm{Fail}\) & \(--\) & \(--\) & \(--\)
		& \(14\) & \(0.0918\) & \(9.6631\mathrm{e}{-07}\) & \(2.5585\mathrm{e}{-04}\)
		& \(9\) & \(0.0585\) & \(2.2326\mathrm{e}{-07}\) & \(5.0331\mathrm{e}{-05}\) \\
		
		\bottomrule
	\end{tabular*}
\end{table}

	Table~\ref{tab:ex4-comparison} reports the results for
	Example~\ref{Ex4} and clearly shows the effect of the diagonal shift.
	When $\mu=0$, the problem is difficult, and none of the methods succeeds
	over the full range of problem sizes. Increasing $\mu$ improves the
	convergence of most methods. M3 gives the most accurate solutions,
	whereas RMNSI provides lower computing times than GGS and FPI-SS. For
	$\mu=4$, RMNSI becomes the fastest method for the larger systems while
	maintaining satisfactory accuracy. Thus, the proposed method offers a
	good balance between speed and accuracy when the matrix has sufficient
	diagonal dominance.

	\begin{figure}[width=\linewidth,pos=H]
		\centering
		\includegraphics[width=\linewidth]{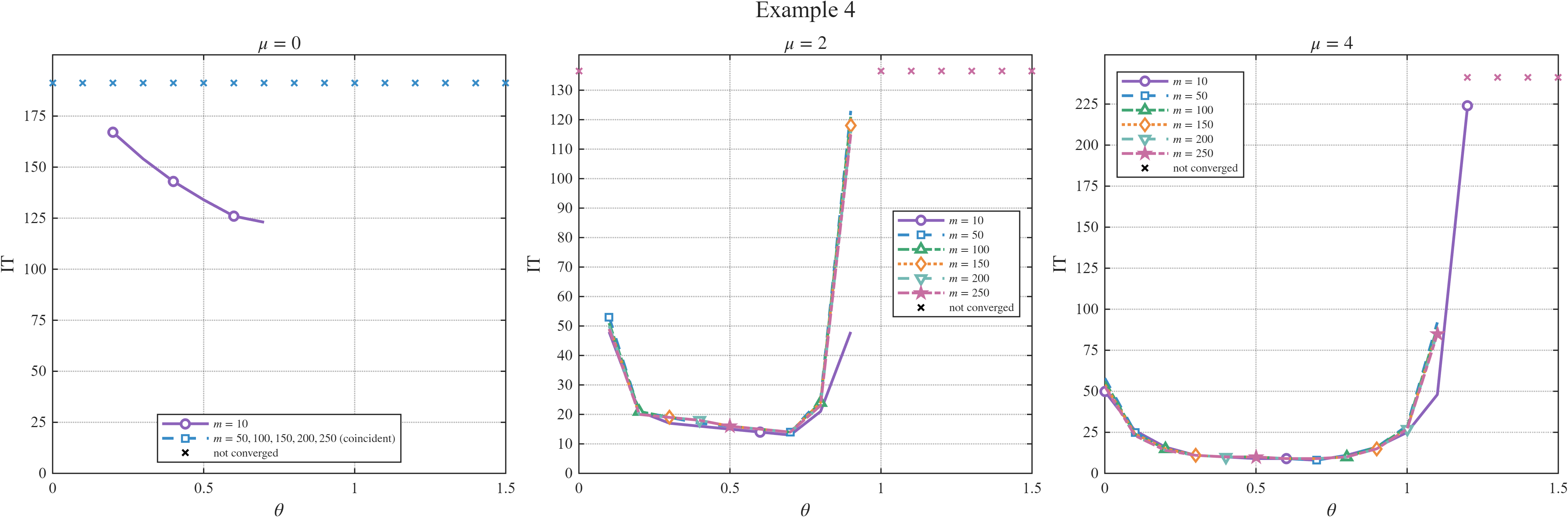}
		\caption{\(\mathrm{Effect~of~the~relaxation~parameter~\theta~on~the~iteration~
				count~of~RMNSI~for~Example~\ref{Ex4}~with~different~values~of~m.}\)}
		\label{fig:theta-ex4}
	\end{figure}

Figure~\ref{fig:theta-ex4} presents the range of $\theta$ for which
RMNSI converges in Example~\ref{Ex4}. This range is narrow without a
diagonal shift but becomes wider and more stable for positive values of
$\mu$; values outside it result in slow convergence or failure.

\begin{example}\label{Ex5} (\cite{LuoEtAl2026}). The matrices $V$ and $D$ are the same as those
	in Example~\ref{Ex4}, but the prescribed complementarity solution is
$
		z_i^*=
		\begin{cases}
			1,  & i\text{ is odd},\\
			10, & i\text{ is even}.
		\end{cases}
$
	Thus, $z^*=(1,10,1,10,\ldots)^{\mathsf T}$. We again set
	$q=-Vz^*$, so that $w^*=0$.
	
	By \eqref{eq:LCP-GAVE-2}, the resulting GAVE~\eqref{eq:GAVE} has
	$A=V+D$, $B=D-V$, and $b=-q$. Thus the exact solution is
$
		x^*=\frac{1}{2}z^*
		=
		\left(
		\frac{1}{2},5,\frac{1}{2},5,\ldots
		\right)^{\mathsf T}.
$

	Compared with Example~\ref{Ex4}, this example has a larger variation
	between consecutive components of the exact solution.
\end{example}

\begin{table}[width=\linewidth,cols=13,pos=h]
	\caption{\(\mathrm{Performance~comparison~of~tested~methods~
			for~Example~\ref{Ex5}.}\)}
	\label{tab:ex5-comparison}
	\footnotesize
	\setlength{\tabcolsep}{1.2pt}
	\renewcommand{\arraystretch}{0.9}
	
	\begin{tabular*}{\tblwidth}
		{@{\extracolsep{\fill}}l*{12}{c}@{}}
		\toprule
		
		\(\mathrm{FPM\mbox{-}M}\)
		& \multicolumn{4}{c}{\(\mu=0\)}
		& \multicolumn{4}{c}{\(\mu=2\)}
		& \multicolumn{4}{c}{\(\mu=4\)} \\
		\cmidrule(lr){2-5}
		\cmidrule(lr){6-9}
		\cmidrule(lr){10-13}
		\(m\)
		& \(\mathrm{IT}\) & \(\mathrm{CPU}\) & \(\mathrm{RES}\) & \(\mathrm{ERR}\)
		& \(\mathrm{IT}\) & \(\mathrm{CPU}\) & \(\mathrm{RES}\) & \(\mathrm{ERR}\)
		& \(\mathrm{IT}\) & \(\mathrm{CPU}\) & \(\mathrm{RES}\) & \(\mathrm{ERR}\) \\
		\midrule
		\(10\)
		& \(\mathrm{Fail}\) & \(--\) & \(--\) & \(--\)
		& \(\mathrm{Fail}\) & \(--\) & \(--\) & \(--\)
		& \(51\) & \(0.0003\) & \(8.3972\mathrm{e}{-07}\) & \(1.5766\mathrm{e}{-05}\) \\
		\(50\)
		& \(\mathrm{Fail}\) & \(--\) & \(--\) & \(--\)
		& \(\mathrm{Fail}\) & \(--\) & \(--\) & \(--\)
		& \(58\) & \(0.0036\) & \(9.9910\mathrm{e}{-07}\) & \(8.8969\mathrm{e}{-05}\) \\
		\(100\)
		& \(\mathrm{Fail}\) & \(--\) & \(--\) & \(--\)
		& \(\mathrm{Fail}\) & \(--\) & \(--\) & \(--\)
		& \(57\) & \(0.0255\) & \(8.6692\mathrm{e}{-07}\) & \(1.5387\mathrm{e}{-04}\) \\
		\(150\)
		& \(\mathrm{Fail}\) & \(--\) & \(--\) & \(--\)
		& \(\mathrm{Fail}\) & \(--\) & \(--\) & \(--\)
		& \(55\) & \(0.0593\) & \(9.6719\mathrm{e}{-07}\) & \(2.5726\mathrm{e}{-04}\) \\
		\(200\)
		& \(\mathrm{Fail}\) & \(--\) & \(--\) & \(--\)
		& \(\mathrm{Fail}\) & \(--\) & \(--\) & \(--\)
		& \(54\) & \(0.1389\) & \(9.7903\mathrm{e}{-07}\) & \(3.4704\mathrm{e}{-04}\) \\
		\(250\)
		& \(\mathrm{Fail}\) & \(--\) & \(--\) & \(--\)
		& \(\mathrm{Fail}\) & \(--\) & \(--\) & \(--\)
		& \(54\) & \(0.1953\) & \(8.8007\mathrm{e}{-07}\) & \(3.8982\mathrm{e}{-04}\) \\
		
		\midrule
		
		\(\mathrm{M3}\)
		& \multicolumn{4}{c}{\(\mu=0\)}
		& \multicolumn{4}{c}{\(\mu=2\)}
		& \multicolumn{4}{c}{\(\mu=4\)} \\
		\cmidrule(lr){2-5}
		\cmidrule(lr){6-9}
		\cmidrule(lr){10-13}
		\(m\)
		& \(\mathrm{IT}\) & \(\mathrm{CPU}\) & \(\mathrm{RES}\) & \(\mathrm{ERR}\)
		& \(\mathrm{IT}\) & \(\mathrm{CPU}\) & \(\mathrm{RES}\) & \(\mathrm{ERR}\)
		& \(\mathrm{IT}\) & \(\mathrm{CPU}\) & \(\mathrm{RES}\) & \(\mathrm{ERR}\) \\
		\midrule
		\(10\)
		& \(23\) & \(0.0005\) & \(3.0905\mathrm{e}{-16}\) & \(5.1114\mathrm{e}{-14}\)
		& \(8\) & \(0.0007\) & \(3.5824\mathrm{e}{-16}\) & \(1.3213\mathrm{e}{-14}\)
		& \(6\) & \(0.0004\) & \(2.4001\mathrm{e}{-16}\) & \(6.3236\mathrm{e}{-15}\) \\
		\(50\)
		& \(73\) & \(0.0116\) & \(3.6915\mathrm{e}{-16}\) & \(3.9837\mathrm{e}{-12}\)
		& \(8\) & \(0.0069\) & \(5.4968\mathrm{e}{-16}\) & \(1.1294\mathrm{e}{-13}\)
		& \(6\) & \(0.0073\) & \(3.9576\mathrm{e}{-16}\) & \(5.2420\mathrm{e}{-14}\) \\
		\(100\)
		& \(151\) & \(0.1121\) & \(3.6700\mathrm{e}{-16}\) & \(6.0160\mathrm{e}{-11}\)
		& \(8\) & \(0.043\) & \(4.4567\mathrm{e}{-16}\) & \(1.7016\mathrm{e}{-13}\)
		& \(6\) & \(0.0362\) & \(3.0738\mathrm{e}{-16}\) & \(8.2165\mathrm{e}{-14}\) \\
		\(150\)
		& \(323\) & \(0.5381\) & \(3.9029\mathrm{e}{-16}\) & \(2.1425\mathrm{e}{-10}\)
		& \(8\) & \(0.0914\) & \(4.4174\mathrm{e}{-16}\) & \(2.5003\mathrm{e}{-13}\)
		& \(6\) & \(0.083\) & \(3.1877\mathrm{e}{-16}\) & \(1.2798\mathrm{e}{-13}\) \\
		\(200\)
		& \(\mathrm{Fail}\) & \(--\) & \(--\) & \(--\)
		& \(8\) & \(0.1944\) & \(4.3593\mathrm{e}{-16}\) & \(3.3114\mathrm{e}{-13}\)
		& \(6\) & \(0.1679\) & \(3.0506\mathrm{e}{-16}\) & \(1.6252\mathrm{e}{-13}\) \\
		\(250\)
		& \(\mathrm{Fail}\) & \(--\) & \(--\) & \(--\)
		& \(8\) & \(0.3002\) & \(4.3493\mathrm{e}{-16}\) & \(4.0876\mathrm{e}{-13}\)
		& \(6\) & \(0.4972\) & \(3.1480\mathrm{e}{-16}\) & \(2.1073\mathrm{e}{-13}\) \\
		
		\midrule
		
		\(\mathrm{GGS}\)
		& \multicolumn{4}{c}{\(\mu=0\)}
		& \multicolumn{4}{c}{\(\mu=2\)}
		& \multicolumn{4}{c}{\(\mu=4\)} \\
		\cmidrule(lr){2-5}
		\cmidrule(lr){6-9}
		\cmidrule(lr){10-13}
		\(m\)
		& \(\mathrm{IT}\) & \(\mathrm{CPU}\) & \(\mathrm{RES}\) & \(\mathrm{ERR}\)
		& \(\mathrm{IT}\) & \(\mathrm{CPU}\) & \(\mathrm{RES}\) & \(\mathrm{ERR}\)
		& \(\mathrm{IT}\) & \(\mathrm{CPU}\) & \(\mathrm{RES}\) & \(\mathrm{ERR}\) \\
		\midrule
		\(10\)
		& \(112\) & \(0.0584\) & \(9.7354\mathrm{e}{-07}\) & \(5.3819\mathrm{e}{-04}\)
		& \(16\) & \(0.0083\) & \(8.2306\mathrm{e}{-07}\) & \(5.3709\mathrm{e}{-05}\)
		& \(11\) & \(0.0061\) & \(7.7898\mathrm{e}{-07}\) & \(3.9218\mathrm{e}{-05}\) \\
		\(50\)
		& \(\mathrm{Fail}\) & \(--\) & \(--\) & \(--\)
		& \(18\) & \(0.6226\) & \(5.7390\mathrm{e}{-07}\) & \(2.0844\mathrm{e}{-04}\)
		& \(12\) & \(0.4029\) & \(5.0497\mathrm{e}{-07}\) & \(1.3299\mathrm{e}{-04}\) \\
		\(100\)
		& \(\mathrm{Fail}\) & \(--\) & \(--\) & \(--\)
		& \(18\) & \(9.2162\) & \(6.2697\mathrm{e}{-07}\) & \(4.5643\mathrm{e}{-04}\)
		& \(12\) & \(6.1703\) & \(5.3911\mathrm{e}{-07}\) & \(2.8421\mathrm{e}{-04}\) \\
		\(150\)
		& \(\mathrm{Fail}\) & \(--\) & \(--\) & \(--\)
		& \(18\) & \(55.954\) & \(6.4473\mathrm{e}{-07}\) & \(7.0439\mathrm{e}{-04}\)
		& \(12\) & \(36.744\) & \(5.5053\mathrm{e}{-07}\) & \(4.3542\mathrm{e}{-04}\) \\
		\(200\)
		& \(\mathrm{Fail}\) & \(--\) & \(--\) & \(--\)
		& \(18\) & \(193.09\) & \(6.5363\mathrm{e}{-07}\) & \(9.5235\mathrm{e}{-04}\)
		& \(12\) & \(127.62\) & \(5.5625\mathrm{e}{-07}\) & \(5.8663\mathrm{e}{-04}\) \\
		\(250\)
		& \(\mathrm{Fail}\) & \(--\) & \(--\) & \(--\)
		& \(18\) & \(555.03\) & \(6.5897\mathrm{e}{-07}\) & \(1.2003\mathrm{e}{-03}\)
		& \(12\) & \(381.78\) & \(5.5968\mathrm{e}{-07}\) & \(7.3784\mathrm{e}{-04}\) \\
		
		\midrule
		
		\(\mathrm{FPI\mbox{-}SS}\)
		& \multicolumn{4}{c}{
			\makecell[c]{\(\mu=0\)\\
				\((\alpha=3.5,\omega=0.75)\)}}
		& \multicolumn{4}{c}{
			\makecell[c]{\(\mu=2\)\\
				\((\alpha=9.9,\omega=0.9)\)}}
		& \multicolumn{4}{c}{
			\makecell[c]{\(\mu=4\)\\
				\((\alpha=13.2,\omega=0.9)\)}} \\
		\cmidrule(lr){2-5}
		\cmidrule(lr){6-9}
		\cmidrule(lr){10-13}
		\(m\)
		& \(\mathrm{IT}\) & \(\mathrm{CPU}\) & \(\mathrm{RES}\) & \(\mathrm{ERR}\)
		& \(\mathrm{IT}\) & \(\mathrm{CPU}\) & \(\mathrm{RES}\) & \(\mathrm{ERR}\)
		& \(\mathrm{IT}\) & \(\mathrm{CPU}\) & \(\mathrm{RES}\) & \(\mathrm{ERR}\) \\
		\midrule
		\(10\)
		& \(80\) & \(0.0008\) & \(8.8087\mathrm{e}{-07}\) & \(7.2999\mathrm{e}{-05}\)
		& \(18\) & \(0.0002\) & \(8.9063\mathrm{e}{-07}\) & \(6.2223\mathrm{e}{-05}\)
		& \(12\) & \(0.0002\) & \(7.5536\mathrm{e}{-07}\) & \(2.7436\mathrm{e}{-05}\) \\
		\(50\)
		& \(\mathrm{Fail}\) & \(--\) & \(--\) & \(--\)
		& \(20\) & \(0.0047\) & \(7.1706\mathrm{e}{-07}\) & \(2.6154\mathrm{e}{-04}\)
		& \(13\) & \(0.0034\) & \(3.7724\mathrm{e}{-07}\) & \(9.4339\mathrm{e}{-05}\) \\
		\(100\)
		& \(\mathrm{Fail}\) & \(--\) & \(--\) & \(--\)
		& \(20\) & \(0.0297\) & \(7.8057\mathrm{e}{-07}\) & \(5.6934\mathrm{e}{-04}\)
		& \(13\) & \(0.0159\) & \(3.8972\mathrm{e}{-07}\) & \(2.0042\mathrm{e}{-04}\) \\
		\(150\)
		& \(\mathrm{Fail}\) & \(--\) & \(--\) & \(--\)
		& \(20\) & \(0.0765\) & \(8.0184\mathrm{e}{-07}\) & \(8.7712\mathrm{e}{-04}\)
		& \(13\) & \(0.055\) & \(3.9391\mathrm{e}{-07}\) & \(3.0652\mathrm{e}{-04}\) \\
		\(200\)
		& \(\mathrm{Fail}\) & \(--\) & \(--\) & \(--\)
		& \(20\) & \(0.17\) & \(8.1250\mathrm{e}{-07}\) & \(1.1849\mathrm{e}{-03}\)
		& \(13\) & \(0.0971\) & \(3.9602\mathrm{e}{-07}\) & \(4.1261\mathrm{e}{-04}\) \\
		\(250\)
		& \(\mathrm{Fail}\) & \(--\) & \(--\) & \(--\)
		& \(20\) & \(0.2592\) & \(8.1890\mathrm{e}{-07}\) & \(1.4927\mathrm{e}{-03}\)
		& \(13\) & \(0.5698\) & \(3.9728\mathrm{e}{-07}\) & \(5.1871\mathrm{e}{-04}\) \\
		
		\midrule
		
		\(\mathbf{RMNSI}\)
		& \multicolumn{4}{c}{
			\makecell[c]{\(\mu=0\)\\ \((\theta=0.7)\)}}
		& \multicolumn{4}{c}{
			\makecell[c]{\(\mu=2\)\\ \((\theta=0.7)\)}}
		& \multicolumn{4}{c}{
			\makecell[c]{\(\mu=4\)\\ \((\theta=0.7)\)}} \\
		\cmidrule(lr){2-5}
		\cmidrule(lr){6-9}
		\cmidrule(lr){10-13}
		\(m\)
		& \(\mathrm{IT}\) & \(\mathrm{CPU}\) & \(\mathrm{RES}\) & \(\mathrm{ERR}\)
		& \(\mathrm{IT}\) & \(\mathrm{CPU}\) & \(\mathrm{RES}\) & \(\mathrm{ERR}\)
		& \(\mathrm{IT}\) & \(\mathrm{CPU}\) & \(\mathrm{RES}\) & \(\mathrm{ERR}\) \\
		\midrule
		\(10\)
		& \(104\) & \(0.0011\) & \(9.7805\mathrm{e}{-07}\) & \(5.4774\mathrm{e}{-04}\)
		& \(13\) & \(0.0002\) & \(4.9603\mathrm{e}{-07}\) & \(3.4688\mathrm{e}{-05}\)
		& \(8\) & \(0.0001\) & \(4.8641\mathrm{e}{-07}\) & \(2.4217\mathrm{e}{-05}\) \\
		\(50\)
		& \(\mathrm{Fail}\) & \(--\) & \(--\) & \(--\)
		& \(14\) & \(0.0017\) & \(6.2234\mathrm{e}{-07}\) & \(2.2670\mathrm{e}{-04}\)
		& \(9\) & \(0.001\) & \(2.1594\mathrm{e}{-07}\) & \(5.6465\mathrm{e}{-05}\) \\
		\(100\)
		& \(\mathrm{Fail}\) & \(--\) & \(--\) & \(--\)
		& \(14\) & \(0.0115\) & \(6.8079\mathrm{e}{-07}\) & \(4.9628\mathrm{e}{-04}\)
		& \(9\) & \(0.0066\) & \(2.3095\mathrm{e}{-07}\) & \(1.2135\mathrm{e}{-04}\) \\
		\(150\)
		& \(\mathrm{Fail}\) & \(--\) & \(--\) & \(--\)
		& \(14\) & \(0.0293\) & \(7.0037\mathrm{e}{-07}\) & \(7.6585\mathrm{e}{-04}\)
		& \(9\) & \(0.0222\) & \(2.3597\mathrm{e}{-07}\) & \(1.8623\mathrm{e}{-04}\) \\
		\(200\)
		& \(\mathrm{Fail}\) & \(--\) & \(--\) & \(--\)
		& \(14\) & \(0.0822\) & \(7.1018\mathrm{e}{-07}\) & \(1.0354\mathrm{e}{-03}\)
		& \(9\) & \(0.0393\) & \(2.3848\mathrm{e}{-07}\) & \(2.5110\mathrm{e}{-04}\) \\
		\(250\)
		& \(\mathrm{Fail}\) & \(--\) & \(--\) & \(--\)
		& \(14\) & \(0.0958\) & \(7.1608\mathrm{e}{-07}\) & \(1.3050\mathrm{e}{-03}\)
		& \(9\) & \(0.1728\) & \(2.3999\mathrm{e}{-07}\) & \(3.1598\mathrm{e}{-04}\) \\
		
		\bottomrule
	\end{tabular*}
\end{table}

Table~\ref{tab:ex5-comparison} summarizes the results for
Example~\ref{Ex5}. Without the diagonal shift, the problem is difficult
for all the methods, and none completes the full set of tests. For
$\mu=2$, all methods except FPM-M converge; M3 gives the highest
accuracy, while RMNSI has the shortest running time. When $\mu=4$, all
methods converge, and RMNSI remains the fastest across the tested problem
sizes. Although M3 produces more accurate approximations, RMNSI gives a
better practical balance between computing time and solution accuracy.

\begin{figure}[width=\linewidth,pos=h]
	\centering
	\includegraphics[width=\linewidth]{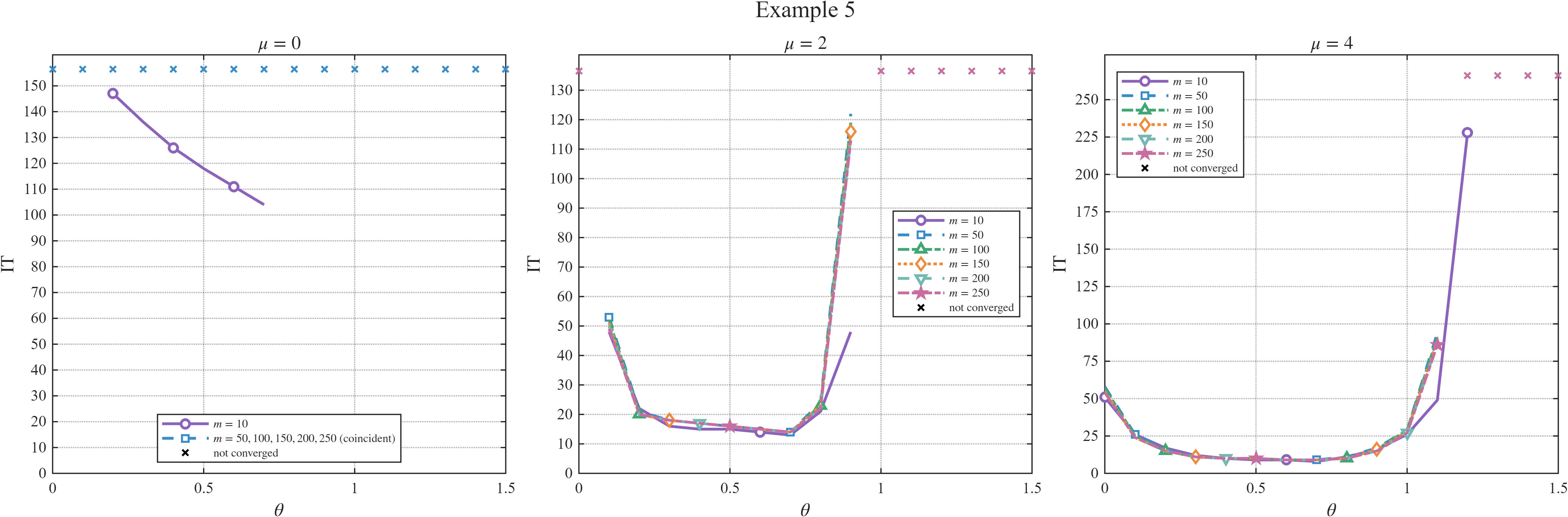}
	\caption{\(\mathrm{Effect~of~the~relaxation~parameter~\theta~on~the~iteration~
			count~of~RMNSI~for~Example~\ref{Ex5}~with~different~values~of~m.}\)}
	\label{fig:theta-ex5}
\end{figure}	

Figure~\ref{fig:theta-ex5} records how the relaxation parameter affects
Example~\ref{Ex5}. The poor behaviour of the unshifted problem is
substantially improved by positive values of $\mu$, which create a clear
low-iteration region shared by the different problem sizes.

\begin{example}\label{Ex6}(\cite{XieQiHan2025}). Let
	$L=\mathcal{V}_{m,\mu}(-1,4,-1;-1,-1)$ and define the positive and
	negative parts of $x$ by
	$x_{+}=\max\{x,0\}$ and $x_{-}=\max\{-x,0\}$, respectively. Consider
	\begin{equation}\label{eq:asymmetric-ridge}
		\min_{x\in\mathbb{R}^n}\Phi(x),
		~
		\Phi(x)
		=
		\frac{1}{2}\|Lx-c\|_2^2
		+\lambda_{+}\|x_{+}\|_2^2
		+\lambda_{-}\|x_{-}\|_2^2,
	\end{equation}
	where $c\in\mathbb{R}^n$ and $\lambda_{+},\lambda_{-}>0$. Since
	$x_{+}=(x+|x|)/2$ and $x_{-}=(|x|-x)/2$, the gradient of $\Phi$ is
	\begin{equation*}
		\nabla\Phi(x)
		=
		\left[
		L^{\mathsf T}L
		+(\lambda_{+}+\lambda_{-})I_n
		\right]x
		-(\lambda_{-}-\lambda_{+})|x|
		-L^{\mathsf T}c.
	\end{equation*}
	Consequently, the first-order optimality condition
	$\nabla\Phi(x)=0$ is equivalent to GAVE~\eqref{eq:GAVE} with
	\begin{equation*}
		A=L^{\mathsf T}L
		+(\lambda_{+}+\lambda_{-})I_n,
		\qquad
		B=(\lambda_{-}-\lambda_{+})I_n,
		\qquad
		b=L^{\mathsf T}c.
	\end{equation*}
	
	In the numerical experiment, we choose
	$\lambda_{+}=60$ and $\lambda_{-}=100$, giving
	\begin{equation*}
		A=L^{\mathsf T}L+160I_n,
		\qquad
		B=40I_n.
	\end{equation*}
	The prescribed exact solution is
	\begin{equation*}
		x_i^*
		=
		\begin{cases}
			-1, & i\text{ is odd},\\
			1,  & i\text{ is even},
		\end{cases}
		\qquad
		x^*=(-1,1,-1,1,\ldots)^{\mathsf T}.
	\end{equation*}
	Since $|x^*|=\boldsymbol e$, the right-hand side is constructed as
	$
		b
		=
		Ax^*-B|x^*|
		=
		(L^{\mathsf T}L+160I_n)x^*-40\boldsymbol e.
$
	For the selected matrices, $L$ is nonsingular. We therefore choose
$
		c=(L^{\mathsf T})^{-1}b,
$
	which ensures that $L^{\mathsf T}c=b$.
	
	Because $\lambda_{+}>0$ and $\lambda_{-}>0$, the function $\Phi$ is
	strongly convex. Hence, the condition $\nabla\Phi(x^*)=0$ is both
	necessary and sufficient for optimality, and $x^*$ is the unique
	minimizer of \eqref{eq:asymmetric-ridge}.
\end{example}

\begin{table}[width=\linewidth,cols=13,pos=H]
	\caption{\(\mathrm{Performance~comparison~of~tested~methods~
			for~Example~\ref{Ex6}.}\)}
	\label{tab:ex6-comparison}
	\footnotesize
	\setlength{\tabcolsep}{1.2pt}
	\renewcommand{\arraystretch}{0.9}
	
	\begin{tabular*}{\tblwidth}
		{@{\extracolsep{\fill}}l*{12}{c}@{}}
		\toprule
		
		\(\mathrm{FPM\mbox{-}M}\)
		& \multicolumn{4}{c}{\(\mu=0\)}
		& \multicolumn{4}{c}{\(\mu=2\)}
		& \multicolumn{4}{c}{\(\mu=4\)} \\
		\cmidrule(lr){2-5}
		\cmidrule(lr){6-9}
		\cmidrule(lr){10-13}
		\(m\)
		& \(\mathrm{IT}\) & \(\mathrm{CPU}\) & \(\mathrm{RES}\) & \(\mathrm{ERR}\)
		& \(\mathrm{IT}\) & \(\mathrm{CPU}\) & \(\mathrm{RES}\) & \(\mathrm{ERR}\)
		& \(\mathrm{IT}\) & \(\mathrm{CPU}\) & \(\mathrm{RES}\) & \(\mathrm{ERR}\) \\
		\midrule
		\(10\)
		& \(11\) & \(0.0022\) & \(5.2005\mathrm{e}{-07}\) & \(7.3392\mathrm{e}{-06}\)
		& \(11\) & \(0.0007\) & \(5.0524\mathrm{e}{-07}\) & \(7.2684\mathrm{e}{-06}\)
		& \(10\) & \(0.0008\) & \(8.0784\mathrm{e}{-07}\) & \(1.1643\mathrm{e}{-05}\) \\
		\(50\)
		& \(11\) & \(0.006\) & \(5.3070\mathrm{e}{-07}\) & \(3.7515\mathrm{e}{-05}\)
		& \(11\) & \(0.0211\) & \(5.4606\mathrm{e}{-07}\) & \(3.9415\mathrm{e}{-05}\)
		& \(10\) & \(0.0203\) & \(8.9521\mathrm{e}{-07}\) & \(6.4876\mathrm{e}{-05}\) \\
		\(100\)
		& \(11\) & \(0.0298\) & \(5.3179\mathrm{e}{-07}\) & \(7.5203\mathrm{e}{-05}\)
		& \(11\) & \(0.0974\) & \(5.5114\mathrm{e}{-07}\) & \(7.9593\mathrm{e}{-05}\)
		& \(10\) & \(0.1113\) & \(9.0615\mathrm{e}{-07}\) & \(1.3141\mathrm{e}{-04}\) \\
		\(150\)
		& \(11\) & \(0.1171\) & \(5.3215\mathrm{e}{-07}\) & \(1.1289\mathrm{e}{-04}\)
		& \(11\) & \(0.3792\) & \(5.5283\mathrm{e}{-07}\) & \(1.1977\mathrm{e}{-04}\)
		& \(10\) & \(0.5096\) & \(9.0979\mathrm{e}{-07}\) & \(1.9795\mathrm{e}{-04}\) \\
		\(200\)
		& \(11\) & \(0.182\) & \(5.3232\mathrm{e}{-07}\) & \(1.5057\mathrm{e}{-04}\)
		& \(11\) & \(0.5542\) & \(5.5368\mathrm{e}{-07}\) & \(1.5995\mathrm{e}{-04}\)
		& \(10\) & \(0.6802\) & \(9.1162\mathrm{e}{-07}\) & \(2.6448\mathrm{e}{-04}\) \\
		\(250\)
		& \(11\) & \(0.3143\) & \(5.3242\mathrm{e}{-07}\) & \(1.8826\mathrm{e}{-04}\)
		& \(11\) & \(0.9591\) & \(5.5419\mathrm{e}{-07}\) & \(2.0012\mathrm{e}{-04}\)
		& \(10\) & \(0.667\) & \(9.1271\mathrm{e}{-07}\) & \(3.3102\mathrm{e}{-04}\) \\
		
		\midrule
		
		\(\mathrm{M3}\)
		& \multicolumn{4}{c}{\(\mu=0\)}
		& \multicolumn{4}{c}{\(\mu=2\)}
		& \multicolumn{4}{c}{\(\mu=4\)} \\
		\cmidrule(lr){2-5}
		\cmidrule(lr){6-9}
		\cmidrule(lr){10-13}
		\(m\)
		& \(\mathrm{IT}\) & \(\mathrm{CPU}\) & \(\mathrm{RES}\) & \(\mathrm{ERR}\)
		& \(\mathrm{IT}\) & \(\mathrm{CPU}\) & \(\mathrm{RES}\) & \(\mathrm{ERR}\)
		& \(\mathrm{IT}\) & \(\mathrm{CPU}\) & \(\mathrm{RES}\) & \(\mathrm{ERR}\) \\
		\midrule
		\(10\)
		& \(33\) & \(0.0009\) & \(8.9156\mathrm{e}{-07}\) & \(1.2676\mathrm{e}{-05}\)
		& \(33\) & \(0.0012\) & \(7.4989\mathrm{e}{-07}\) & \(1.0902\mathrm{e}{-05}\)
		& \(32\) & \(0.001\) & \(7.7600\mathrm{e}{-07}\) & \(1.1379\mathrm{e}{-05}\) \\
		\(50\)
		& \(33\) & \(0.0294\) & \(9.0764\mathrm{e}{-07}\) & \(6.4386\mathrm{e}{-05}\)
		& \(33\) & \(0.0244\) & \(7.9011\mathrm{e}{-07}\) & \(5.7544\mathrm{e}{-05}\)
		& \(32\) & \(0.0296\) & \(8.4194\mathrm{e}{-07}\) & \(6.2082\mathrm{e}{-05}\) \\
		\(100\)
		& \(33\) & \(0.1575\) & \(9.0925\mathrm{e}{-07}\) & \(1.2895\mathrm{e}{-04}\)
		& \(33\) & \(0.1657\) & \(7.9469\mathrm{e}{-07}\) & \(1.1576\mathrm{e}{-04}\)
		& \(32\) & \(0.1345\) & \(8.4973\mathrm{e}{-07}\) & \(1.2538\mathrm{e}{-04}\) \\
		\(150\)
		& \(33\) & \(0.5005\) & \(9.0977\mathrm{e}{-07}\) & \(1.9351\mathrm{e}{-04}\)
		& \(33\) & \(0.452\) & \(7.9619\mathrm{e}{-07}\) & \(1.7398\mathrm{e}{-04}\)
		& \(32\) & \(0.4989\) & \(8.5231\mathrm{e}{-07}\) & \(1.8867\mathrm{e}{-04}\) \\
		\(200\)
		& \(33\) & \(0.815\) & \(9.1003\mathrm{e}{-07}\) & \(2.5807\mathrm{e}{-04}\)
		& \(33\) & \(0.7464\) & \(7.9694\mathrm{e}{-07}\) & \(2.3219\mathrm{e}{-04}\)
		& \(32\) & \(0.7709\) & \(8.5360\mathrm{e}{-07}\) & \(2.5196\mathrm{e}{-04}\) \\
		\(250\)
		& \(33\) & \(1.5122\) & \(9.1018\mathrm{e}{-07}\) & \(3.2263\mathrm{e}{-04}\)
		& \(33\) & \(1.3754\) & \(7.9739\mathrm{e}{-07}\) & \(2.9040\mathrm{e}{-04}\)
		& \(32\) & \(1.2502\) & \(8.5437\mathrm{e}{-07}\) & \(3.1525\mathrm{e}{-04}\) \\
		
		\midrule
		
		\(\mathrm{GGS}\)
		& \multicolumn{4}{c}{\(\mu=0\)}
		& \multicolumn{4}{c}{\(\mu=2\)}
		& \multicolumn{4}{c}{\(\mu=4\)} \\
		\cmidrule(lr){2-5}
		\cmidrule(lr){6-9}
		\cmidrule(lr){10-13}
		\(m\)
		& \(\mathrm{IT}\) & \(\mathrm{CPU}\) & \(\mathrm{RES}\) & \(\mathrm{ERR}\)
		& \(\mathrm{IT}\) & \(\mathrm{CPU}\) & \(\mathrm{RES}\) & \(\mathrm{ERR}\)
		& \(\mathrm{IT}\) & \(\mathrm{CPU}\) & \(\mathrm{RES}\) & \(\mathrm{ERR}\) \\
		\midrule
		\(10\)
		& \(6\) & \(0.0046\) & \(1.7268\mathrm{e}{-07}\) & \(1.5619\mathrm{e}{-06}\)
		& \(6\) & \(0.0274\) & \(8.6992\mathrm{e}{-07}\) & \(9.9011\mathrm{e}{-06}\)
		& \(7\) & \(0.0285\) & \(3.8315\mathrm{e}{-07}\) & \(4.9011\mathrm{e}{-06}\) \\
		\(50\)
		& \(6\) & \(0.1743\) & \(1.0460\mathrm{e}{-07}\) & \(5.0042\mathrm{e}{-06}\)
		& \(6\) & \(1.3157\) & \(9.4624\mathrm{e}{-07}\) & \(6.0426\mathrm{e}{-05}\)
		& \(7\) & \(1.58\) & \(5.2264\mathrm{e}{-07}\) & \(3.5537\mathrm{e}{-05}\) \\
		\(100\)
		& \(5\) & \(2.0335\) & \(9.7756\mathrm{e}{-07}\) & \(1.0829\mathrm{e}{-04}\)
		& \(6\) & \(15.685\) & \(9.4986\mathrm{e}{-07}\) & \(1.2337\mathrm{e}{-04}\)
		& \(7\) & \(20.677\) & \(5.3959\mathrm{e}{-07}\) & \(7.3796\mathrm{e}{-05}\) \\
		\(150\)
		& \(5\) & \(11.232\) & \(9.1171\mathrm{e}{-07}\) & \(1.5617\mathrm{e}{-04}\)
		& \(6\) & \(33.185\) & \(9.5080\mathrm{e}{-07}\) & \(1.8631\mathrm{e}{-04}\)
		& \(7\) & \(73.607\) & \(5.4522\mathrm{e}{-07}\) & \(1.1205\mathrm{e}{-04}\) \\
		\(200\)
		& \(5\) & \(39.154\) & \(8.7624\mathrm{e}{-07}\) & \(2.0384\mathrm{e}{-04}\)
		& \(6\) & \(172.68\) & \(9.5121\mathrm{e}{-07}\) & \(2.4924\mathrm{e}{-04}\)
		& \(7\) & \(92.778\) & \(5.4804\mathrm{e}{-07}\) & \(1.5031\mathrm{e}{-04}\) \\
		\(250\)
		& \(5\) & \(315.84\) & \(8.5402\mathrm{e}{-07}\) & \(2.5141\mathrm{e}{-04}\)
		& \(6\) & \(384.79\) & \(9.5144\mathrm{e}{-07}\) & \(3.1217\mathrm{e}{-04}\)
		& \(7\) & \(342.47\) & \(5.4973\mathrm{e}{-07}\) & \(1.8857\mathrm{e}{-04}\) \\
		
		\midrule
		
		\(\mathrm{FPI\mbox{-}SS}\)
		& \multicolumn{4}{c}{
			\makecell[c]{\(\mu=0\)\\
				\((\alpha=30,\omega=0.2)\)}}
		& \multicolumn{4}{c}{
			\makecell[c]{\(\mu=2\)\\
				\((\alpha=29.5,\omega=0.2)\)}}
		& \multicolumn{4}{c}{
			\makecell[c]{\(\mu=4\)\\
				\((\alpha=30,\omega=0.15)\)}} \\
		\cmidrule(lr){2-5}
		\cmidrule(lr){6-9}
		\cmidrule(lr){10-13}
		\(m\)
		& \(\mathrm{IT}\) & \(\mathrm{CPU}\) & \(\mathrm{RES}\) & \(\mathrm{ERR}\)
		& \(\mathrm{IT}\) & \(\mathrm{CPU}\) & \(\mathrm{RES}\) & \(\mathrm{ERR}\)
		& \(\mathrm{IT}\) & \(\mathrm{CPU}\) & \(\mathrm{RES}\) & \(\mathrm{ERR}\) \\
		\midrule
		\(10\)
		& \(48\) & \(0.0018\) & \(9.6360\mathrm{e}{-07}\) & \(8.6612\mathrm{e}{-06}\)
		& \(54\) & \(0.0029\) & \(9.9917\mathrm{e}{-07}\) & \(8.2521\mathrm{e}{-06}\)
		& \(61\) & \(0.0045\) & \(8.4564\mathrm{e}{-07}\) & \(9.9204\mathrm{e}{-06}\) \\
		\(50\)
		& \(47\) & \(0.0198\) & \(9.8879\mathrm{e}{-07}\) & \(4.7097\mathrm{e}{-05}\)
		& \(53\) & \(0.1087\) & \(9.0177\mathrm{e}{-07}\) & \(4.0471\mathrm{e}{-05}\)
		& \(59\) & \(0.1223\) & \(9.6738\mathrm{e}{-07}\) & \(6.4557\mathrm{e}{-05}\) \\
		\(100\)
		& \(47\) & \(0.1363\) & \(9.5805\mathrm{e}{-07}\) & \(9.2210\mathrm{e}{-05}\)
		& \(53\) & \(0.5519\) & \(8.2655\mathrm{e}{-07}\) & \(7.6078\mathrm{e}{-05}\)
		& \(59\) & \(0.9032\) & \(9.2526\mathrm{e}{-07}\) & \(1.2793\mathrm{e}{-04}\) \\
		\(150\)
		& \(47\) & \(0.471\) & \(9.4732\mathrm{e}{-07}\) & \(1.3728\mathrm{e}{-04}\)
		& \(52\) & \(1.6314\) & \(9.9393\mathrm{e}{-07}\) & \(1.3603\mathrm{e}{-04}\)
		& \(59\) & \(2.5225\) & \(9.1030\mathrm{e}{-07}\) & \(1.9127\mathrm{e}{-04}\) \\
		\(200\)
		& \(47\) & \(0.9128\) & \(9.4186\mathrm{e}{-07}\) & \(1.8235\mathrm{e}{-04}\)
		& \(52\) & \(2.8186\) & \(9.7614\mathrm{e}{-07}\) & \(1.7908\mathrm{e}{-04}\)
		& \(59\) & \(1.4468\) & \(9.0263\mathrm{e}{-07}\) & \(2.5460\mathrm{e}{-04}\) \\
		\(250\)
		& \(47\) & \(4.0541\) & \(9.3855\mathrm{e}{-07}\) & \(2.2741\mathrm{e}{-04}\)
		& \(52\) & \(4.345\) & \(9.6526\mathrm{e}{-07}\) & \(2.2211\mathrm{e}{-04}\)
		& \(59\) & \(2.0085\) & \(8.9797\mathrm{e}{-07}\) & \(3.1793\mathrm{e}{-04}\) \\
		
		\midrule
		
		\(\mathbf{RMNSI}\)
		& \multicolumn{4}{c}{
			\makecell[c]{\(\mu=0\)\\ \((\theta=1.6)\)}}
		& \multicolumn{4}{c}{
			\makecell[c]{\(\mu=2\)\\ \((\theta=1.6)\)}}
		& \multicolumn{4}{c}{
			\makecell[c]{\(\mu=4\)\\ \((\theta=1.5)\)}} \\
		\cmidrule(lr){2-5}
		\cmidrule(lr){6-9}
		\cmidrule(lr){10-13}
		\(m\)
		& \(\mathrm{IT}\) & \(\mathrm{CPU}\) & \(\mathrm{RES}\) & \(\mathrm{ERR}\)
		& \(\mathrm{IT}\) & \(\mathrm{CPU}\) & \(\mathrm{RES}\) & \(\mathrm{ERR}\)
		& \(\mathrm{IT}\) & \(\mathrm{CPU}\) & \(\mathrm{RES}\) & \(\mathrm{ERR}\) \\
		\midrule
		\(10\)
		& \(3\) & \(0.0015\) & \(2.3749\mathrm{e}{-07}\) & \(3.1381\mathrm{e}{-06}\)
		& \(3\) & \(0.0004\) & \(9.5916\mathrm{e}{-07}\) & \(1.1904\mathrm{e}{-05}\)
		& \(3\) & \(0.0004\) & \(4.0725\mathrm{e}{-07}\) & \(4.7518\mathrm{e}{-06}\) \\
		\(50\)
		& \(3\) & \(0.0026\) & \(1.9577\mathrm{e}{-07}\) & \(1.3534\mathrm{e}{-05}\)
		& \(3\) & \(0.0115\) & \(5.0534\mathrm{e}{-07}\) & \(3.2243\mathrm{e}{-05}\)
		& \(3\) & \(0.0161\) & \(1.9581\mathrm{e}{-07}\) & \(1.1455\mathrm{e}{-05}\) \\
		\(100\)
		& \(3\) & \(0.0163\) & \(1.8822\mathrm{e}{-07}\) & \(2.6296\mathrm{e}{-05}\)
		& \(3\) & \(0.0492\) & \(3.9550\mathrm{e}{-07}\) & \(5.1655\mathrm{e}{-05}\)
		& \(3\) & \(0.0586\) & \(1.4056\mathrm{e}{-07}\) & \(1.6508\mathrm{e}{-05}\) \\
		\(150\)
		& \(3\) & \(0.0555\) & \(1.8556\mathrm{e}{-07}\) & \(3.9036\mathrm{e}{-05}\)
		& \(3\) & \(0.1818\) & \(3.5021\mathrm{e}{-07}\) & \(6.9710\mathrm{e}{-05}\)
		& \(3\) & \(0.2629\) & \(1.1589\mathrm{e}{-07}\) & \(2.0491\mathrm{e}{-05}\) \\
		\(200\)
		& \(3\) & \(0.0957\) & \(1.8419\mathrm{e}{-07}\) & \(5.1770\mathrm{e}{-05}\)
		& \(3\) & \(0.3033\) & \(3.2498\mathrm{e}{-07}\) & \(8.7254\mathrm{e}{-05}\)
		& \(3\) & \(0.1345\) & \(1.0121\mathrm{e}{-07}\) & \(2.3946\mathrm{e}{-05}\) \\
		\(250\)
		& \(3\) & \(0.5126\) & \(1.8336\mathrm{e}{-07}\) & \(6.4502\mathrm{e}{-05}\)
		& \(3\) & \(0.5386\) & \(3.0878\mathrm{e}{-07}\) & \(1.0454\mathrm{e}{-04}\)
		& \(3\) & \(0.2065\) & \(9.1227\mathrm{e}{-08}\) & \(2.7074\mathrm{e}{-05}\) \\
		
		\bottomrule
	\end{tabular*}
\end{table}

	Table~\ref{tab:ex6-comparison} gives the results for the ridge-regression
	problem in Example~\ref{Ex6}. All the methods converge for every tested
	case, but RMNSI requires the least work and is generally the fastest.
	Its iteration count remains nearly unchanged as the problem size increases.
	Moreover, its accuracy improves as $\mu$ increases and is particularly
	strong for $\mu=4$. Although GGS also uses relatively few iterations,
	its running time grows rapidly for large systems. These results show
	that RMNSI is also effective for an application-based GAVE.

	\begin{figure}[width=\linewidth,pos=H]
		\centering
		\includegraphics[width=\linewidth]{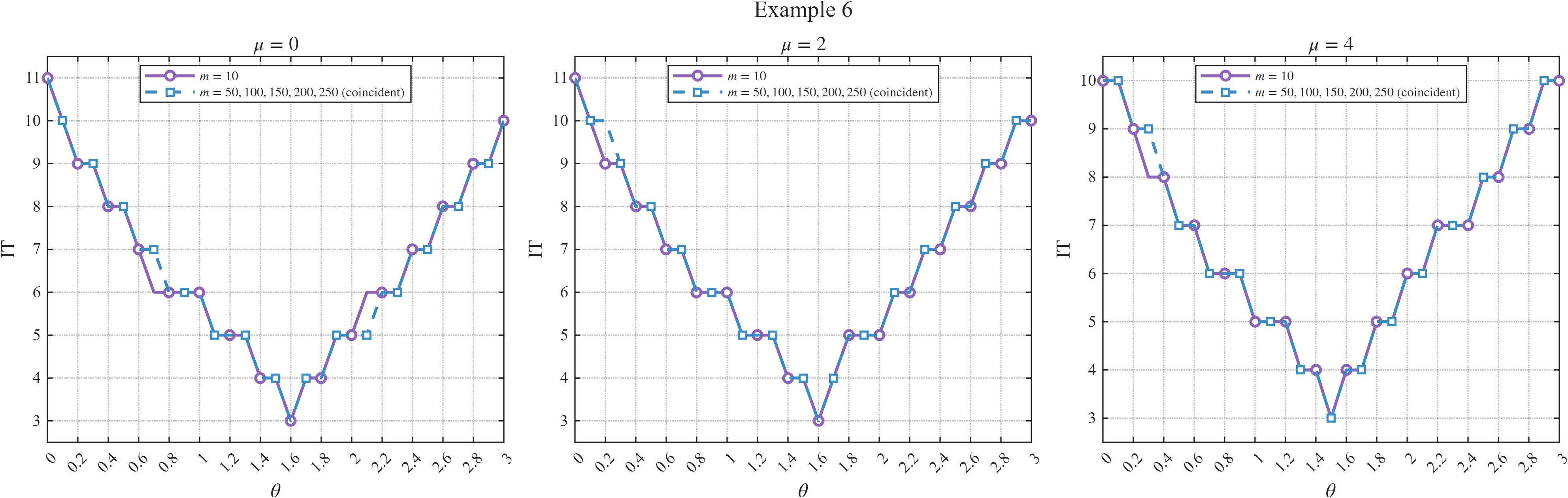}
		\caption{\(\mathrm{Effect~of~the~relaxation~parameter~\theta~on~the~iteration~
				count~of~RMNSI~for~Example~\ref{Ex6}~with~different~values~of~m.}\)}
		\label{fig:theta-ex6}
	\end{figure}

	For the ridge-regression problem in Example~\ref{Ex6},
	Figure~\ref{fig:theta-ex6} depicts a broad valley-shaped response to
	$\theta$. Its shape changes only slightly with $\mu$, and the close
	curves indicate that the preferred parameter is largely unaffected by
	the matrix dimension.
	
Overall, the numerical results show that a small iteration count does
not necessarily correspond to a low computational cost. Although GGS
is parameter-free and often converges in relatively few iterations for
positive values of $\mu$, its CPU time grows considerably with the
problem dimension, and it fails in several unshifted cases. FPI-SS is competitive for some positively shifted problems, but its performance depends on the joint choice of $\alpha$ and $\omega$ and it performs less reliably in several tests with $\mu=0$. M3 performs particularly well for the tested problems whose exact solutions have a uniform sign pattern. FPM-M is effective in several shifted cases,
but it requires additional auxiliary quantities and does not perform
uniformly across all the test problems. In comparison, RMNSI provides
consistent convergence over the tested matrix structures, shifts,
problem dimensions, and solution sign patterns while requiring only
the relaxation parameter $\theta$.


\section{Conclusions and future directions}\label{sec6}

A relaxed maximum-based normal $S$-iteration method has been developed
for solving generalized absolute value equations. By combining the
maximum-based reformulation with a two-stage fixed-point process, the
selection of an auxiliary matrix is avoided while both under-relaxed
and over-relaxed updates are permitted. Moreover, the two linear
systems arising at each iteration have the same coefficient matrix
$A+B$, so that one factorization can be reused throughout the
computation.

Based on the theoretical analysis and numerical experiments, the
following conclusions can be drawn:
\begin{enumerate}
	\item The condition
	$\chi_\theta=\eta(|1-\theta|+\theta\eta)<1$ guarantees both unique
	solvability of the GAVE and global convergence of RMNSI from an
	arbitrary initial vector. Under this condition, $R$-linear convergence
	and residual- and iteration-based error estimates are obtained.
	Norm-based, $M$-matrix, and diagonal-dominance conditions provide more
	directly verifiable alternatives.
	
	\item The global bound is minimized at $\theta=1$, but this value does
	not necessarily produce the smallest iteration count. Once the sign
	pattern near a solution is fixed, the local behaviour is governed by a
	sharper iteration matrix. This explains why over-relaxed values of
	$\theta$ can perform well and why convergence can be observed even when
	the global sufficient condition is not satisfied.
	
	\item The numerical results show that RMNSI provides a favourable
	balance among iteration count, computational time, and accuracy for
	the tested complementarity-derived, dense, and ridge-regression
	problems. The preferred value of $\theta$ is found to depend mainly on
	the matrix structure and the diagonal shift, while its dependence on
	the problem dimension is generally weak. Positive shifts enlarge the
	observed convergence range in several test problems, whereas only a
	small change is observed for the ridge-regression example.
\end{enumerate}

The present analysis assumes that $A+B$ is nonsingular and, for the
local result, that the solution has no zero components. Future work
will consider adaptive selection of $\theta$, local analysis for
solutions containing zero components, and implementations based on
inexact linear solves for large-sparse systems. Extensions to
rectangular GAVEs, generalized absolute value matrix equations, and
related complementarity problems will also be investigated.

\section*{Acknowledgements}
The first author acknowledges the support of the Ministry of Education, Government of India, through the GATE fellowship (Registration No. MA22S25023148).

\printcredits

	\bibliographystyle{elsarticle-num}
	\bibliography{cas-refs}

\end{document}